\documentclass[reqno,a4paper]{amsart} 
\usepackage[utf8]{inputenc}
\usepackage[T1]{fontenc}
\usepackage[british]{babel}
\usepackage{amsfonts, amssymb, amsmath}
\usepackage[all]{xy}
\usepackage{hyperref}
\usepackage{mathbbol}
\usepackage{calrsfs}
\usepackage{mathabx}
\usepackage{xcolor}
\usepackage{scalerel}
 
\theoremstyle{plain}
\newtheorem{theorem}[subsection]{Theorem}
\newtheorem{proposition}[subsection]{Proposition}
\newtheorem{corollary}[subsection]{Corollary}	
\newtheorem{lemma}[subsection]{Lemma}

\theoremstyle{definition}
\newtheorem{definition}[subsection]{Definition}

\newtheorem{convention}[subsection]{Convention}

\theoremstyle{remark}
\newtheorem{remark}[subsection]{Remark}
\newtheorem{example}[subsection]{Example}
\newtheorem{examples}[subsection]{Examples}

\numberwithin{equation}{section}

\newcommand{\defn}{\textbf}
\newcommand{\noproof}{\hfil\qed}
\newcommand{\join}{\vee}
\newcommand{\meet}{\wedge}
\newcommand{\ot}{\leftarrow}

\newcommand{\comp}{\circ}
\newcommand\ttop{\scaleobj{0.6}{\top}}
\newcommand\pperp{\scaleobj{0.6}{\perp}}

\DeclareMathOperator{\Eq}{Eq}
\DeclareMathOperator{\I}{I}
\DeclareMathOperator{\U}{U}
\DeclareMathOperator{\Ii}{I_{1}}
\DeclareMathOperator{\Triv}{Triv}

\newcommand{\B}{\mathcal{B}}
\newcommand{\C}{\mathcal{C}}

\newcommand{\E}{\mathcal{E}}
\newcommand{\X}{\mathcal{X}}
\newcommand{\Y}{\mathcal{Y}}

\newcommand{\Arr}{\mathsf{Arr}}
\newcommand{\Ext}{\mathsf{Ext}}
\newcommand{\Cov}{\mathsf{Cov}}
\newcommand{\TCov}{\mathsf{TCov}}

\newcommand{\Set}{\mathsf{Set}}

\newcommand{\N}{\mathbb{N}}

\newdir{>}{{}*:(1,-.2)@^{>}*:(1,+.2)@_{>}}
\newdir{<}{{}*:(1,+.2)@^{<}*:(1,-.2)@_{<}}
\newdir{>>}{{}*!/3.5pt/:(1,-.2)@^{>}*!/3.5pt/:(1,+.2)@_{>}*!/7pt/:(1,-.2)@^{>}*!/7pt/:(1,+.2)@_{>}}
\newdir{ >}{{}*!/-8pt/@{>}}

\def\pullback{% with thanks to Valerian Even
 \ar@{-}[]+R+<6pt,-1pt>;[]+RD+<6pt,-6pt>%
 \ar@{-}[]+D+<1pt,-6pt>;[]+RD+<6pt,-6pt>}

 \def\pullbackleft{%
 \ar@{-}[]+L+<-6pt,-1pt>;[]+LD+<-6pt,-6pt>%
 \ar@{-}[]+D+<-1pt,-6pt>;[]+LD+<-6pt,-6pt>}

\def\rotpullbackdots{%
 \ar@{.}[]+D+<6pt,-6pt>;[]+D+<0pt,-12pt>;%
 \ar@{.}[]+D+<0pt,-12pt>;[]+D+<-6pt,-6pt>}

 \def\pullbackdots{%
 \ar@{.}[]+R+<6pt,-1pt>;[]+RD+<6pt,-6pt>%
 \ar@{.}[]+D+<1pt,-6pt>;[]+RD+<6pt,-6pt>}

\def\skewpullback{%
 \ar@{-}[]+R+<6pt,-1pt>;[]+RD+<-3pt,-6pt>%
 \ar@{-}[]+D+<-9pt,-6pt>;[]+RD+<-3pt,-6pt>}

\def\altskewpullback{%
 \ar@{-}[]+R+<6pt,-1pt>;[]+RD+<-1pt,-6pt>%
 \ar@{-}[]+D+<-6pt,-6pt>;[]+RD+<-1pt,-6pt>}

\def\skewpullbackdots{%
 \ar@{.}[]+R+<6pt,-1pt>;[]+RD+<-3pt,-6pt>%
 \ar@{.}[]+D+<-9pt,-6pt>;[]+RD+<-3pt,-6pt>}

\def\opskewpullbackdots{%
 \ar@{.}[]+D+<-1pt,-6pt>;[]+LD+<-6pt,-13pt>%
 \ar@{.}[]+LD+<-6pt,-6pt>;[]+LD+<-6pt,-13pt>}

 \def\opskewpullback{%
 \ar@{-}[]+D+<-1pt,-6pt>;[]+LD+<-6pt,-13pt>%
 \ar@{-}[]+LD+<-6pt,-6pt>;[]+LD+<-6pt,-13pt>}

 \def\splitpullback{%
\ar@{-}[]+R+<6pt,-1pt>;[]+RD+<6pt,-6pt>%
\ar@{-}[]+D+<.51ex,-6pt>;[]+RD+<6pt,-6pt>}

\def\opsplitpullback{%
\ar@{-}[]+R+<6pt,-.51ex>;[]+RD+<6pt,-6pt>%
\ar@{-}[]+D+<1pt,-6pt>;[]+RD+<6pt,-6pt>}

\begin{document}

\title[A symmetric approach to higher coverings]{A symmetric approach to higher coverings\\ in categorical Galois theory}

\author{Fara\ Renaud}
\author{Tim\ Van~der Linden}

\email{tim.vanderlinden@uclouvain.be}
\email{fara.renaud@uclouvain.be}

\address{Institut de Recherche en Math\'ematique et Physique, Universit\'e catholique de Louvain, che\-min du cyclotron~2 bte~L7.01.02, B--1348 Louvain-la-Neuve, Belgium}

\thanks{The first author was supported by the Fonds de la Recherche Scientifique--FNRS under grant number MIS-F.4536.19 and the \emph{Formation à la Recherche dans l'Industrie et dans l'Agriculture (FRIA)}. The second author is a Senior Research Associate of the Fonds de la Recherche Scientifique--FNRS}

\begin{abstract}
	In the context of a tower of (strongly Birkhoff) Galois structures in the sense of categorical Galois theory, we show that the concept of a higher covering admits a characterisation which is at the same time \emph{absolute} (with respect to the base level in the tower), rather than inductively defined relative to extensions of a lower order; and \emph{symmetric}, rather than depending on a perspective in terms of arrows pointing in a certain chosen direction. This result applies to the Galois theory of quandles, for instance, where it helps us characterising the higher coverings in purely algebraic terms.
\end{abstract}


\maketitle

\tableofcontents

\section{Introduction}

The aim of this article is two-fold: first of all, we wish to expound a symmetric approach to the concept of a \emph{higher covering} from categorical Galois theory~\cite{Janelidze:Double,EGVdL,Ever2010,Mathieu}. Our approach, however, depends on the development of a setting which is sufficiently strong, while remaining general enough to be applicable to our concrete examples. Developing this setting is the second main purpose of this text.

It is well known that, despite its inductive definition, the notion of a \emph{higher extension} is entirely symmetric: as shown in~\cite{EGoeVdL}, there is no dependence on the choice of direction. For higher \emph{coverings}, however, which in their inductive definition are viewed as arrows between arrows with a specific direction, the symmetry is much more difficult to explain. There is the combinatorial proof of~\cite{Tomasthesis,Ever2010} which is valid in semi-abelian~\cite{Janelidze-Marki-Tholen} categories, but which is hard to extend to weaker settings, essentially because it depends on an ad-hoc commutator theory which does not generalise as such to the field of application we have in mind. Indeed, our original motivation is to describe higher coverings of quandles and racks---a setting which is very far from being semi-abelian, since the categories involved are neither pointed nor Mal'tsev.

Instead, we here present a new conceptual proof of the symmetry of higher coverings, based on an idea in~\cite{RVdL2}, which is not only simpler, but also valid in a much wider context---including the categories of racks and quandles, where it helps us characterising the coverings in algebraic terms~\cite{Ren2021c}. In that setting, it is noticeable that unlike coverings, higher \emph{normal} and \emph{trivial} coverings are not symmetrical, and these facts are clearly reflected in certain steps of the general proof not being valid for those classes of extensions---the key difference being that coverings are reflected by pullbacks along extensions, whereas normal coverings and trivial coverings are not.

This first goal of the article forms the subject of Section~\ref{Section Discrete Fibrations Double}--\ref{Section Higher Coverings}. The theory in this latter part of the article is developed in the context of a so-called \emph{tower of strongly Birkhoff Galois structures}. Our second goal is precisely to properly describe the construction of such a tower in a context which is sufficiently general to include not only all classical examples worked out in a semi-abelian or an exact Mal'tsev context, but also our leading examples of racks and quandles. This endeavour---the development of a basic higher Galois theory, valid in non-Mal'tsev contexts---takes place in the first half of the present article, in Section~\ref{Section Preliminaries}--\ref{Intermediate results}.

We give a brief overview of the main results in the text. Section~\ref{Section Preliminaries} recalls well-known concepts such as categories of higher arrows and extensions, as well as (towers of) Galois structures from the literature. In Section~\ref{Section Calculus} we establish key stability properties, which are usually considered in a Mal'tsev context, in a setting which is sufficiently general for our purposes. This involves properties of pullbacks and pushouts, effective descent, etc. In Section~\ref{Intermediate results} we work towards our first concrete aim, namely a collection of results allowing us to show that under certain convenient conditions, a Galois structure which satisfies the strong Birkhoff condition induces an entire tower of Galois structures. This involves: a higher-dimensional version of the well-known equivalence between the Birkhoff condition and closedness under quotients of the class of coverings (Proposition~\ref{PropositionBirkhoffIffClosedUnderQuotient}); a result on stability under quotients along double extensions of the class of coverings (Proposition~\ref{CentralExtensionsAreStableAlongDoubleExtensions}); and a construction of the reflector from extensions to coverings (Proposition~\ref{PropReflctOfCentralExtensions}); together with an explanation why all this is applicable both in the classical Mal'tsev context and in the examples of racks and quandles.

We change gears in Section~\ref{Section Discrete Fibrations Double} where we start the development of a theory of \emph{symmetric coverings}, which allows a non-inductive treatment of the concept of a covering. The basic idea is very simple and based on the concept of a \emph{discrete fibration}, which is what we call an $n$-fold extension which, considered as an $n$-cube, is a limit diagram---the easiest example being a pullback square of regular epimorphisms. Recall (details are given in the text) that, in a Galois structure $\Gamma = (\Y,\X,L,\I,\eta,\epsilon,\E)$ involving an adjunction $L\dashv \I$, an extension $c$ as in the diagram
\[
	\vcenter{\xymatrix@!0@=4em{ \I L(T) \ar[d]_-{\I L(t)} & T \pullback \pullbackleft \ar[d]^-{t} \ar[l]_-{\eta_{T}} \ar[r] & A \ar[d]^-{c} \\
	\I L(E) & E \ar[l]^-{\eta_{E}} \ar[r]_-{e} & B
	}}
\]
is a \emph{covering} if and only if there is an extension $e$ such that the pullback of $c$ along~$e$ is a \emph{trivial covering} $t$, which means that the reflection square on the left is a pullback. Equivalently, $t$ is a pullback of some extension $x$ in the given reflective subcategory~$\X$ of the ground category $\Y$. In what follows, such an $x$ is called a \emph{primitive covering}. In other words, $c$ is a covering when a span of discrete fibrations $x\ot t\to c$ exists in $\Y$, where $x$ is a primitive covering---an extension whose domain and codomain are objects in the full subcategory $\X$ of $\Y$. This idea applies at arbitrary levels of a given tower of Galois structures; it suffices to replace the pullback squares with higher-order discrete fibrations as defined above. Our main result here is Theorem~\ref{Main Theorem}, which says that in the context of an appropriate tower of Galois structures, this approach does indeed characterise the inductively defined coverings of Galois theory---an $n$-fold extension $\alpha$ is a higher-dimensional covering if and only if there exists a span of discrete fibrations ${\beta\ot \tau\to \alpha}$ of order $n+1$ where $\beta$ is an $n$-cubical extension in the base category $\X$---thus showing that the concept of a higher covering is symmetric: it does not depend on the way the $n$-cube is considered as an arrow between $(n-1)$-cubes.

Section~\ref{Section Discrete Fibrations Double} introduces and studies discrete fibrations of order three, which are used in Section~\ref{Section Double Coverings} in the symmetric characterisation of double coverings (Theorem~\ref{Theorem Covering iff Symmetric Covering}). In Section~\ref{Section Higher Coverings} this is extended to arbitrary degrees, which eventually leads to the above-mentioned Theorem~\ref{Main Theorem}.

\section{Preliminaries}\label{Section Preliminaries}

\subsection{Higher morphisms}\label{SubsectionHigherMorphisms}
In order to study higher extensions and higher categorical Galois theory, we need an appropriate description of \emph{higher morphisms} (or \emph{higher arrows}) in a category $\Y$. We refer to~\cite{EGoeVdL} for a detailed presentation.

Higher morphisms are traditionally defined inductively as arrows between arrows (between arrows\dots). Given a category $\Y$, the \defn{arrow category} is $\Arr(\Y)$, whose objects are arrows (morphisms) in $\Y$. A morphism $\alpha\colon{f_A \to f_B}$ in such a category of morphisms is given by a pair of morphisms in $\Y$, the \defn{top} and \defn{bottom components} of $\alpha$ which we denote $\alpha = (\alpha_{\ttop},\alpha_{\pperp})$ and which give rise to a commutative diagram~\eqref{EquationDoubleArrow}.
\begin{equation}\label{EquationDoubleArrow}
	\vcenter{\xymatrix@!0@=4em{
	A_{\ttop} \ar[r]^-{\alpha_{\ttop}} \ar[d]_-{f_A} & B_{\ttop} \ar[d]^-{f_B} \\
	A_{\pperp} \ar[r]_-{\alpha_{\pperp}} & B_{\pperp}
	}}
\end{equation}
For $n\geq 2$, an \defn{$n$-arrow} or \defn{$n$-fold arrow} in $\Y$ is then an object of the category $\Arr^n(\Y)$, which is defined inductively as $\Arr(\Arr^{n-1}(\Y))$. We are also interested in an $n$-fold arrow's underlying commutative diagram in the base category $\Y$, which we call an \defn{$n$-cubical diagram} in $\Y$.

Given an $n$-cubical diagram $\alpha$, its \defn{initial object} $\alpha_{\meet}$ is the initial object of the given diagram, when it is viewed as a subcategory of $\Y$. Dually, the \defn{terminal object} $\alpha_{\join}$ is the terminal object of that category. Given a morphism of $n$-arrows $\phi\colon \alpha \to \beta$, the \defn{initial component} of $\phi$ is, in the underlying $(n+1)$-cubical diagram of $\phi$, the morphism with domain $\alpha_{\meet}$ and codomain $\beta_{\meet}$. In general, what we call the \defn{$\Y$-components} of $\phi$ is the collection, within the underlying $(n+1)$-cubical diagram, of the $2^n$ arrows between the $n$-cubical diagrams of $\alpha$ and $\beta$.

We consider all the categories of higher arrows above a given category $\C_0$ simultaneously. The sequence $(\Arr^n(\C_0))_{n\in \N}$ is called the \defn{tower of higher arrows} above $\C_0$, and each category $\Arr^n(\C_0)$ a \defn{level} in this tower. Our aim is to study contexts in which a Galois theory at the base level $\C_0$ induces Galois theories in higher dimensions $\Arr^n(\C_0)$, at each level of a similar tower of higher morphisms called \emph{extensions}~\cite{EGVdL,Ever2010}.

\subsection{Extensions and higher extensions}\label{Subsection Ext and higher}
Categorical Galois theory is defined relatively to a chosen class of morphisms called \defn{extensions} (which are in some sense parallel to the fibrations in homotopy theory). Given a class of extensions at the base level $\C_0$, there is a convenient induced class of extensions at each level~$\Arr^n(\C_0)$. In this article we write $\E_1$ for the chosen class of extensions in $\C_0$ and always choose it to be the class of regular epimorphisms. (Note that the index is $1$ rather than $0$, because it consists of arrows rather than objects.) By induction we define a class of extensions $\E_{n+1}$ at each level~$\Arr^n(\C_0)$.

Given a category $\Y$ with a chosen class of morphisms $\E$, we consider a commutative square as in~\eqref{EquationDoubleExtension}, together with the pullback $A_{\pperp} \times_{B_{\pperp}} B_{\ttop}$ of $\alpha_{\pperp}$ and $f_B$.
\begin{equation}\label{EquationDoubleExtension}
	\vcenter{\xymatrix@!0@C=2.5em@R=2.5em{
	A_{\ttop} \ar[rrr]^-{\alpha_{\ttop}} \ar[rd]|(.45){p} \ar[ddd]_{f_A} & & & B_{\ttop} \ar[ddd]^-{f_B} \\
	& A_{\pperp} \times_{B_{\pperp}} B_{\ttop} \ar[rru]_-{\pi_2} \ar[ldd]^-{\pi_1} \\
	\\
	A_{\pperp} \ar[rrr]_-{\alpha_{\pperp}} & & & B_{\pperp}
	}}
\end{equation}
The outer commutative square in this diagram (which, remember, can also be viewed as a morphism $\alpha\colon{f_A \to f_B}$ in $\Arr(\Y)$) is called an \defn{$\E$--double extension}~\cite{Janelidze:Double} when all morphisms in the diagram are in the class $\E$---including the square's so-called \defn{comparison map} $p$.

\begin{example}
	If $\E$ is the class of regular epimorphisms in a regular category, then an $\E$--double extension is a \emph{regular pushout} in the sense of~\cite{Bou2003}. These play a key role in the characterisation of exact Mal'tsev categories amongst regular categories: see Proposition~5.4 in~\cite{Carboni-Kelly-Pedicchio}.
\end{example}

Back to the tower of higher arrows above $\C_0$. Let us assume by induction that $n$-fold extensions are defined and form a class $\E_n$ of morphisms in $\Arr^{n-1}(\C_0)$. A morphism in $\Arr^n(\C_0)$ is called an \defn{$(n+1)$-fold extension} if its underlying diagram in $\Arr^{n-1}(\C_0)$ is an $\E_n$-double extension. The resulting class of $(n+1)$-fold extensions is denoted $\E_{n+1}$.

We then define the \defn{tower of extensions} over the base category $\C_0$ to be the sequence of categories $(\C_n)_{n\in \N}$, where $\C_n \coloneq \Ext^n(\C_0)$ is the full subcategory of~$\Arr^n(\C_0)$ determined by the class of $n$-fold extensions $\E_n$. A category $\C_n$ in this sequence is called the $n$-th \defn{level} of the tower of extensions. The elements of the class $\E_{n+1}$ can either be viewed as the objects of $\Ext^{n+1}(\C_0)$ or as the extensions in $\Ext^n(\C_0)$. Depending on the value of $n$, we can even view them as the double extensions in $\Ext^{n-1}(\C_0)$, etc.

As for higher arrows, we are also interested in an $n$-fold extension's underlying commutative diagram in the base category $\C_0$, which we call an \defn{$n$-cubical extension} in $\C_0$. In Section~\ref{Section Discrete Fibrations Double}, we also need the more general concept of \defn{$n$-cubical extension} in $\C_k$ which is the underlying commutative diagram in $\C_k$, of an $(n+k)$-fold extension where $k\geq0$.

In general, given a category $\Y$, we write $\Arr(B)$ for the slice category over a given object $\beta$ of $\Y$ (often denoted $(\Y\downarrow B)$ or $\Y/_B$ in the literature). For a chosen class of morphisms $\E$, we also write $\Ext(B)$ for the full subcategory of $\Arr(B)$ whose objects are in $\E$.

\begin{example}\label{ExampleHigherExtensions}
	Given a pair of $2$-fold extensions $\gamma$ and $\alpha$ (i.e., $\E_1$-double extensions, objects in $\C_2$), a morphism $(\sigma,\beta)\colon{\gamma \to \alpha}$ between these is given by a $2$-cubical diagram in $\C_1$ (on the left in Figure~\ref{FigureThreeFoldExtension}) or equivalently a $3$-cubical diagram in $\C_0$ (on the right),
	\begin{figure}
		$\vcenter{\xymatrix@!0@=3em{f_C \ar[rrr]^{\sigma = (\sigma_{\ttop},\sigma_{\pperp})} \ar[rd]^-{\pi} \ar[ddd]|-{\gamma = (\gamma_{\ttop},\gamma_{\pperp})} & & & f_A \ar[ddd]|{\alpha=(\alpha_{\ttop},\alpha_{\pperp})} \\
			& f_P \ar[rru]|-{} \ar[ldd]|-{} \\
			\\
			f_D \ar[rrr]_{\beta = (\beta_{\ttop},\beta_{\pperp})} & & & f_B
			}}
			\qquad \qquad
			\vcenter{\xymatrix@!0@=3em{
			& C_{\ttop} \ar[rr]^-{\sigma_{\ttop}} \ar[dd]|(.25){\gamma_{\ttop}}|-{\hole} \ar[ld]_-{f_C} && A_{\ttop} \ar[dd]^-{\alpha_{\ttop}} \ar[ld]|-{f_A} \\
			C_{\pperp} \ar[rr]^(.75){\sigma_{\pperp}} \ar[dd]_-{\gamma_{\pperp}} && A_{\pperp} \ar[dd]^(.25){\alpha_{\pperp}} \\
			& D_{\ttop} \ar[ld]_-{f_D} \ar[rr]^(.25){\beta_{\ttop}}|-{\hole} && B_{\ttop} \ar[ld]^-{f_B} \\
			D_{\pperp} \ar[rr]_-{\beta_{\pperp}} && B_{\pperp}}}$
		\caption{A $3$-fold extension}\label{FigureThreeFoldExtension}
	\end{figure}
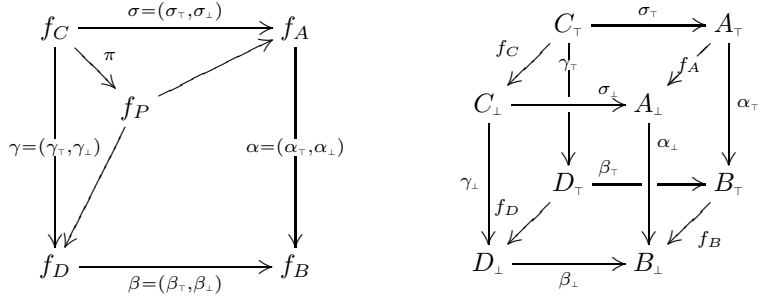
	where $f_P$ is the pullback of $\alpha$ and $\beta$. It is a \emph{$3$-fold extension} if $\sigma$, $\beta$ and the comparison map $\pi = (\pi_{\ttop},\pi_{\pperp})$ are also $2$-fold extensions. That is to say, a $3$-fold extension is the same thing as an $\E_2$-double extension. The initial object of $(\sigma,\beta)$ is $C_{\ttop}$ and the terminal object is $B_{\pperp}$. The initial component of $(\sigma,\beta)\colon \gamma \to \alpha$ is $\sigma_{\ttop}$. The $\C_0$-components of $(\sigma,\beta)\colon \gamma \to \alpha$ are $\sigma_{\ttop}$, $\sigma_{\pperp}$, $\beta_{\ttop}$ and $\beta_{\pperp}$.
\end{example}

\subsection{Categorical Galois theory}\label{Subsection Categorical Galois Theory}
From now on, we assume that our base level $\C_0$ is a Barr exact category~\cite{BaGrOs1971}, and (as before) $\E_1$ is the class of regular epimorphisms in $\C_0$ (the ($1$-fold) extensions).

Given an arbitrary category $\Y$ with chosen class of extensions $\E$, we define an \defn{$\E$-reflective subcategory} $\X$ of $\Y$ to be a full (replete) subcategory with inclusion functor $\I \colon \X \to \Y$ which admits a left adjoint $L \colon {\Y \to \X}$ with unit $\eta$, counit $\epsilon$, and such that for each object $Y$ in $\Y$, the component $\eta_Y\colon Y \to \I L (Y)$ of the unit is in $\E$. We omit the functor $\I$ if confusion is unlikely.

Throughout, we let $\B_0$ be an $\E_1$-reflective subcategory of $\C_0$, with inclusion functor $\I_0 \colon \B_0 \to \C_0$ and left adjoint $F_0\colon { \C_0 \to \B_0}$ whose unit is $\eta^0$ and whose counit is $\epsilon^0$.

In general, an $\E$-reflector $L \colon \Y \to \X$ (or equivalently the entire adjunction, or the reflective subcategory $\X$) is said to satisfy the \defn{$\E$-Birkhoff condition} when given an extension $f\colon{A\to B}$ in $\Y$, the \defn{$L$-reflection square at $f$}, i.e., the morphism $(\eta_A,\eta_B) \colon f \to \I L (f)$ in $\Ext(\Y)$, or more precisely its underlying $2$-cubical diagram in $\Y$
\begin{equation}\label{DiagramReflectionSquare}
	\vcenter{\xymatrix@!0@=4em{
	A \ar[r]^-{\eta_A} \ar[d]_{f} & \I L(A) \ar[d]^-{\I L(f)} \\
	B \ar[r]_-{\eta_B} & \I L(B)
	}}
\end{equation}
is a pushout square in $\Y$. Finally, such an adjunction $L \dashv \I$ (or equivalently the reflector $L$, or the reflective subcategory $\X$) is said to be \defn{strongly $\E$-Birkhoff} if for each extension $f\colon A \to B$ in $\Y$, the reflection square at $f$ (as in Diagram~\ref{DiagramReflectionSquare}) is an $\E$--double extension.

We investigate examples of towers $(\C_n)_{n\in \N}$ where such an adjunction $F_0 \dashv \I_0$ at the base level $\C_0$ induces a similar adjunction $F_n \dashv \I_n$ at each level $\C_n$. We will for instance show that, given such a tower of adjunctions, if $F_n \dashv \I_n$ satisfies the strong $\E_{n+1}$-Birkhoff condition, then $F_{n+1} \dashv \I_{n+1}$ satisfies the $\E_{n+2}$-Birkhoff condition. In certain concrete examples---see Section~\ref{Intermediate results}---we will then be able to show that $F_{n+1} \dashv \I_{n+1}$ is in fact \emph{strongly} $\E_{n+2}$-Birkhoff.

At the base level $\C_0$ (which is Barr exact by assumption), it is known~\cite{JK1994} that~$\B_0$ satisfies the $\E_1$-Birkhoff condition if and only if $\B_0$ is closed under quotients along extensions in~$\C_0$. A similar result holds in higher dimensions---see Proposition~\ref{PropositionBirkhoffIffClosedUnderQuotient}.

It was shown in~\cite[Proposition 2.6]{EGVdL} that in a suitably general context where $L \dashv \I$ is strongly $\E$-Birkhoff we obtain a so-called \emph{admissible Galois structure}: in essence, we have all we need to do Galois theory.

\begin{convention}\label{ConventionGaloisStructure} For our purposes, a \defn{Galois structure} (see~\cite{JanComo1991})
	\[
		\Gamma \coloneq (\Y,\X,L,\I,\eta,\epsilon,\E),
	\]
	is the data of the inclusion $\I$ of an $\E$-reflective subcategory $\X$ in a category $\Y$, with left adjoint $L\colon{\Y \to \X}$, unit $\eta$, counit $\epsilon$ and a class $\E$ of extensions chosen amongst the morphisms of $\Y$. The class $\E$ is subject to the following conditions:
	\begin{enumerate}
		\item $\E$ contains all isomorphisms, and $\E$ is closed under composition;
		\item the image of an extension by the reflector $L$ yields an extension;
		\item pullbacks along extensions exist, and a pullback of an extension is an extension.
	\end{enumerate}
	In the present article, $\E$ will always be a class of regular epimorphisms. Finally, taking pullbacks along extensions should be a ``well behaved algebraic operation'', i.e., we require our extensions to be of \emph{effective $\E$-descent} in $\Y$ (see~\cite{JanTho1994,JanSoTho2004} and Section~\ref{SectionBeyondBarrExactness} below).
\end{convention}

We study examples of towers $(\C_n)_{n \in \N}$ that carry an admissible Galois structure (in the sense of~\cite{Jan1990}) at each level. By the above it suffices that ``all adjunctions are strongly Birkhoff'', which is a more convenient condition to work with in this inductive context.

In general, given an admissible (or strongly Birkhoff) Galois structure $ \Gamma$ as in Convention~\ref{ConventionGaloisStructure}, a \defn{trivial $\Gamma$-covering} is an extension $t\colon T \to E$ whose $L$-reflection square (on the left below) is a pullback.
\begin{equation*}\label{DiagramDefinitionTrivialCovering}
	\vcenter{\xymatrix@!0@=4em{ L(T) \ar[d]_-{L(t)} & T \pullback \pullbackleft \ar[d]^-{t} \ar[l]_-{\eta_{T}} \ar[r] & A \ar[d]^-{c} \\
	L(E) & E \ar[l]^-{\eta_{E}} \ar[r]_-{e} & B
	}}
\end{equation*}
Equivalently~\cite{JK1994}, the extension $t$ is a pullback of some \defn{primitive $\Gamma$-covering}~$x$, which is an extension that lies in the given full reflective subcategory~$\X$ of the ground category $\Y$.

A \defn{$\Gamma$-covering} (sometimes called \defn{central extension}) is an extension $c\colon A \to B$ for which there is another extension $e\colon E \to B$ such that the pullback $t$ of $c$ along~$e$ is a trivial covering. We then say that $c$ is \defn{split by $e$}. In other words, $c$ is a $\Gamma$-covering when a span of discrete fibrations $x\ot t\to c$ exists in $\Y$, where $x$ is a primitive $\Gamma$-covering. (See, for instance, Figure~1 in~\cite{Ren2020}.) We define $\Cov(\Y)$ to be the full subcategory of $\Ext(\Y)$ determined by the coverings. Finally, a \defn{normal $\Gamma$-covering} $f$ is such that the projections of its kernel pair are trivial coverings, i.e., $f$ is split by itself.

Note that, as a consequence of admissibility, these classes of extensions are all closed under pullbacks along extensions---see for instance~\cite{JK1994}. In fact, admissibility essentially amounts to the condition that the inclusion
\[
	\xymatrix@=3em{\Ext(\Y) \ar@<1ex>@{.>}[r]^-{\Triv} \ar@{}[r]|-{\bot} & \TCov(\Y) \ar@<1ex>@{->}[l]^-{\supseteq}}
\]
of the category $\TCov(\Y)$ of trivial coverings into the category of extensions admits a left adjoint ($\Triv$).

The class of coverings measures a ``sphere of influence'' of $\X$ in $\Y$ and plays an important role in diverse areas of mathematics (for different examples of such an adjunction). Loosely speaking, the \emph{Fundamental Theorem} of categorical Galois theory classifies coverings via actions of a suitable \emph{Galois groupoid} in $\Y$.

\subsection{A tower of Galois structures}\label{Subsection Tower}

We can now explain how a Galois structure
\[
	\Gamma_0 \coloneq (\C_0,\B_0,F_0,\I_0,\eta^0,\epsilon^0,\E_1),
\]
at the base level $\C_0$ may induce a Galois structure $\Gamma_n$ at each level of the tower $(\C_n)_{n\in \N}$. The idea is that the inclusion $\I_0 \colon \B_0 \to \C_0$ induces a concept of $\Gamma_0$-covering. This, in turn, induces the inclusion $\Ii \colon \Cov(\C_0) \to \C_1=\Ext(\C_0)$ of the category of coverings $\Cov(\C_0)$ in the category of extensions $\C_1$. This inclusion functor~$\Ii$ does not share all the properties of $\I_0$. In particular, $\C_1$ is no longer Barr exact or even regular~\cite{Bar1971} in general: see the comment preceding Definition~3.4 in~\cite{EGVdL} and~\cite[Remark~2.0.1]{Ren2021}. In this sense, the class of regular epimorphisms is not appropriate for applying Galois theory in higher dimensions; instead, we use (higher) extensions. Indeed, pullbacks along higher extensions are computed component-wise, thanks to which many of the convenient properties which hold in the Barr exact $\C_0$ can be used to derive similar results in the higher-dimensional context---see Section~\ref{Section Calculus}.

The strong $\E_1$-Birkhoff condition on $F_0 \dashv \I_0$ then implies for instance that $\Cov(\C_0)$ is closed under quotients along double extensions in $\C_1$ (Subsection~\ref{SubsectionStronglyBirkhoffAndCoveringsClosedAlongExtensions}). It is often~\cite{JK:Reflectiveness} the case that $\Cov(\C_0)$ is in turn $\E_2$-reflective in $\C_1$ (Subsection~\ref{SubsectionAlgebraicExamples}). By the equivalence between ``closedness under quotients'' and the Birkhoff property (Subsection~\ref{SubsectionClosureByQuotientsAndBirkhoff}), the reflection $F_1\dashv \I_1$ is $\E_2$-Birkhoff. For it to be strongly Birkhoff, we currently have no general result. However, in concrete examples such as those presented in Subsection~\ref{Racks Quandles}, this derives easily from the permutability of the \emph{centralisation relation}, which is the universal relation to be divided out of an extension for it to become a covering.

When the adjunction $F_1\dashv \I_1$ is indeed strongly $\E_2$-Birkhoff, then we have a Galois structure
\[
	\Gamma_1 \coloneq (\C_1,\B_1,F_1,\I_1,\eta^1,\epsilon^1,\E_2)
\]
and the corresponding concept of a $\Gamma_1$-covering. We can then iterate the above process for the inclusion $\I_2\colon \Cov(\C_1)\to \C_2=\Ext(\C_1)$ and find a Galois structure~$\Gamma_2$. By proceeding inductively, we obtain a Galois structure $\Gamma_n$ for each $n$. Again, there is no guarantee that this works in general, but it works in many examples, and as we shall see in Section~\ref{Intermediate results} below, intermediate results exist that simplify the process showing this.

\begin{convention}
	Throughout this article, the categories in the adjunction at the base level of a tower of Galois structures are denoted $(\C_0,\B_0)$, with class of extensions $\E_1$. An arbitrary fixed level is written $(\C,\B)$, with class of extensions~$\E$: it is~$(\C_n,\B_n)$, with class of extensions $\E_{n+1}$, for some $n$. A generic reflective subcategory, eventually part of a Galois structure $\Gamma$ as in~\ref{ConventionGaloisStructure}, is denoted $\X\leq \Y$.
\end{convention}

\section{Calculus and properties of higher extensions}\label{Section Calculus}

Working in categories of higher morphisms is not easy. Thanks to the concept of a higher extension, we can reduce a problem or computation in $\C_n$ to a couple of problems and computations in $\C_{n-1}$, using the projections on the top and bottom components. Ideally, such a problem or computation boils all the way down to a problem or computation in the base category $\C_0$. For instance:

\begin{examples}
	\begin{enumerate}
		\item The initial object of $\C_n$ is the $n$-cube, all of whose vertices are the initial object of $\C_0$, and all of whose edges are identities.
		\item Given a parallel pair of morphisms $\alpha$, $\beta \colon f_A \rightrightarrows f_B$ in $\C_n$, if the coequaliser of the top component exists in $\C_{n-1}$, and if the pushout of this coequaliser along $f_B$ exists, then this pushout describes the coequaliser of $\alpha$ and $\beta$ in~$\C_n$.
		\item Monomorphisms in $\C_n$ are morphisms for which the top component is a monomorphism in $\C_{n-1}$. By induction, monomorphisms in $\C_n$ are morphisms such that their initial component (see Subsection~\ref{SubsectionHigherMorphisms}) is a monomorphism in $\C_0$.
	\end{enumerate}
\end{examples}

Observe that Proposition 3.5 and Lemma 3.8 from~\cite{EGVdL} easily generalise to our context (as remarked in~\cite{EGoeVdL}: see Example~1.11, Proposition~1.6 and Remark~1.7 there). Note that the following result holds when $\C_0$ is merely regular:
\begin{lemma}\label{LemmaHigherExtensionsCalculus}
	For $n>1$, let $\C_n$ be a level in the tower induced by $\C_0$.
	\begin{enumerate}
		\item If a morphism $\alpha = (\alpha_{\ttop},\alpha_{\pperp})$ in $\C_n$ is such that $\alpha_{\ttop}$ is an $n$-fold extension and $\alpha_{\pperp}$ an isomorphism, then $\alpha$ is an $(n+1)$-fold extension.
		\item The $(n+1)$-fold extensions are closed under composition.
		\item Pullbacks of arbitrary morphisms along $(n+1)$-fold extensions (exist in $\C_n$ and) are computed \emph{component-wise}. Moreover, any pullback of an $(n+1)$-fold extension is an $(n+1)$-fold extension.\noproof
	\end{enumerate}
\end{lemma}

Given a pair of morphisms $\alpha = (\alpha_{\ttop},\alpha_{\pperp})\colon{f_A \to f_B}$ and $\beta = (\beta_{\ttop},\beta_{\pperp})\colon{f_C \to f_B}$ in $\C_n$, their \defn{component-wise pullback}, if it exists, is given by the commutative diagram depicted in Figure~\ref{DiagramComponentWisePullback}
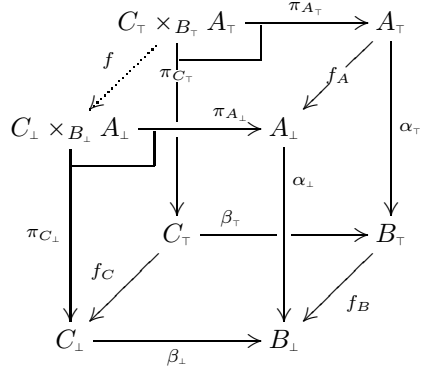
\begin{figure}
	$\xymatrix@!0@=4em{
		& C_{\ttop} \times_{B_{\ttop}} A_{\ttop} \pullback \ar[rr]^-{\pi_{A_{\ttop}}} \ar[dd]|(.25){\pi_{C_{\ttop}}}|-{\hole} \ar@{{}{..}{>}}[ld]_-{f} && A_{\ttop} \ar[dd]^-{\alpha_{\ttop}} \ar[ld]|-{f_A} \\
		C_{\pperp} \times_{B_{\pperp}} A_{\pperp} \pullback \ar[rr]^(.75){\pi_{A_{\pperp}}} \ar[dd]_-{\pi_{C_{\pperp}}} && A_{\pperp} \ar[dd]^(.25){\alpha_{\pperp}} \\
		& C_{\ttop} \ar[ld]_-{f_C} \ar[rr]^(.25){\beta_{\ttop}}|-{\hole} && B_{\ttop} \ar[ld]^-{f_B} \\
		C_{\pperp} \ar[rr]_-{\beta_{\pperp}} && B_{\pperp}}$
	\caption{The component-wise pullback of a pair of double arrows}\label{DiagramComponentWisePullback}
\end{figure}
whose front and back faces are pullbacks in~$\C_{n-1}$. Provided that $\alpha$ is an $(n+1)$-fold extension, Lemma~\ref{LemmaHigherExtensionsCalculus} above says that the component-wise pullback exists (it is given by the pullback in the arrow category~$\Arr^{n}(\C_0)$), moreover $f$ is an $n$-fold extension, and the pullback of $\alpha$ and $\beta$ in $\C_n$ is given by $f$ together with the projections $(\pi_{A_{\ttop}},\pi_{A_{\pperp}})$ and $(\pi_{C_{\ttop}},\pi_{C_{\pperp}})$, where the latter is actually an $(n+1)$-fold extension. In particular, the kernel pair of an $(n+1)$-fold extension $\alpha=(\alpha_{\ttop},\alpha_{\pperp})$ exists in $\C_n$ and is given by the kernel pairs (in $\C_{n-1}$) of $\alpha_{\ttop}$ and $\alpha_{\pperp}$ in each component (together with the induced morphism between those---see Diagram~\ref{DiagramBarrKock} below). Moreover, the legs of such kernel pairs are themselves $(n+1)$-fold extensions---see also Proposition~\ref{PropositionBarrKock}.

Note that when $\C_0$ is a Barr exact \emph{Mal'tsev} category~\cite{CLP1991}, double extensions are the same as pushout squares of extensions, and, as a rule, higher extensions are easier to identify---especially when split epimorphisms are involved. Unlike that stronger context, in our present setting it is not necessarily the case that a split epimorphism of $n$-fold extensions is an $(n+1)$-fold extension. The lack of such arguments is a challenge in this more general context where $\C_0$ need not be a Mal'tsev category.

As we may conclude from~\cite[Proposition 3.3]{EGoeVdL} and the fact that our categories are not Mal'tsev, the axiom (E4) of ``right cancellation'' considered there (see also~\cite[Lemma 3.8]{EGVdL}) cannot hold here. We have the following weaker version:

\begin{lemma}\label{LemmaWeakRightCancellation}
	If the composite $\beta \comp \alpha \colon f_A \to f_C$ is an $(n+1)$-fold extension, and $\alpha \colon f_A \to f_B$ is a commutative square of $n$-fold extensions in $\C_{n-1}$, then $\beta$ is an $(n+1)$-fold extension.
\end{lemma}
\begin{proof}
	Since $\beta_{\ttop} \comp \alpha_{\ttop}$ and $\beta_{\pperp} \comp \alpha_{\pperp}$ are $n$-fold extensions, by induction $\beta$ is a square of $n$-fold extensions in $\C_{n-1}$. Consider the commutative diagram in Figure~\ref{DiagramCompositeOfDExtension}
	\begin{figure}
		$\xymatrix@!0@=3em{
			A_1 \ar[rrrr]^{\alpha_{\ttop}} \ar[rd]^{\phi_1} \ar[dddd]_{f_A} & & & & B_1 \ar[rd]^{\phi_2} \ar[dddd]|(.33){\hole}_(.6){f_B} \ar[rrrr]^{\beta_{\ttop}} & & & & C_1 \ar[dddd]^{f_C}\\
			& P_1 \ar[rd]^{\phi_3} \ar[rrru]|-{\pi_2} \ar[lddd]|-{\pi_1}& & & & P_2 \ar[rrru]_-{p_2} \ar[lddd]^-{p_1} &&&\\
			& & P_3 \ar[lldd]^-{q_1} \ar[rrru]_(.3){q_2} & & & & & & \\
			\\
			A_0 \ar[rrrr]_{\alpha_{\pperp}} & & & & B_0 \ar[rrrr]_{\beta_{\pperp}} & & & & C_0
			}$
		\caption{Weak right cancellation}\label{DiagramCompositeOfDExtension}
	\end{figure}
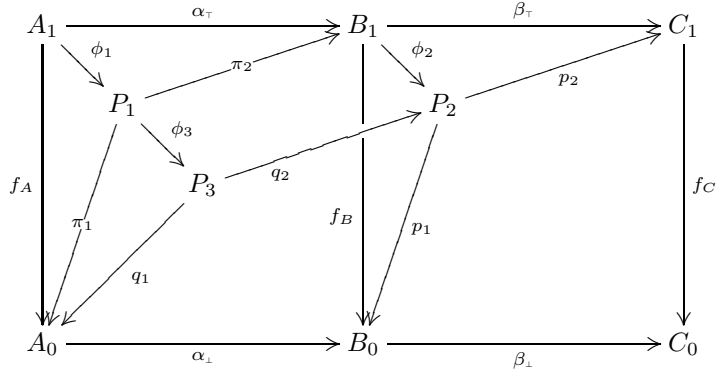
	where $P_1$ is the pullback of $\alpha_{\pperp}$ and $f_B$, $P_2$ is the pullback of $\beta_{\pperp}$ and $f_C$, and $P_3$ is the pullback of $\alpha_{\pperp}$ and the projection $p_1$. Note that since $P_3$ is also a pullback of $\beta_{\pperp}\comp \alpha_{\pperp}$ and~$f_C$, the comparison map (recall Diagram~\ref{EquationDoubleExtension}) of the composite square is $\phi_3\comp \phi_1$---which is an $n$-fold extension. By induction, $\phi_3$ is thus also an $n$-fold extension, and so is the composite $q_2 \comp\phi_3 = \phi_2\comp \pi_2$. We conclude that the comparison map $\phi_2$ of $\beta$ is an $n$-fold extension as well.
\end{proof}

In particular, we deduce that pullbacks along extensions reflect extensions:

\begin{lemma}\label{LemmaPullbacksReflectDExtensions}
	Given a morphism $\alpha = (\alpha_{\ttop},\alpha_{\pperp})$ in $\C_n$, if its pullback along an $(n+1)$-fold extension $\beta$ yields a $(n+1)$-fold extension $\alpha'$, then $\alpha$ is itself a $(n+1)$-fold extension.\noproof
\end{lemma}

\subsection{Beyond Barr exactness: effective descent along double extensions}\label{SectionBeyondBarrExactness}
We are now going to explain that extensions have similar convenient exactness properties to regular epimorphisms in Barr exact categories.

Given a Galois structure $\Gamma$ as in Convention~\ref{ConventionGaloisStructure}, $\Gamma$-coverings, which are the key concept of study, are defined as those extensions $c\colon A \to B$ for which there is another extension $e\colon E \to B$ such that the pullback $t$ of $c$ along $e$ is a trivial $\Gamma$-covering. In most references~\cite{BorJan1994,GraRo2004,Jan2008}, $e$ is further required to be of \emph{effective descent} or \emph{effective $\E$-descent} (see~\cite{JanTho1994,JanSoTho2004}). Such extensions are sometimes also called \emph{monadic extensions}~\cite{JanComo1991,Ever2014}. In the contexts of interest for this work, we shall always have that all our extensions are of effective $\E$-descent, which is why we may use our simplified definition of a $\Gamma$-covering. The idea is to ask that ``pulling back along~$e$ is an \emph{algebraic} operation'', which plays a key role in the proof of the Fundamental Theorem of categorical Galois theory---see for instance~\cite[Corollary~5.4]{JanComo1991}.

Given an $(n+1)$-fold extension $e\colon{E \to B}$ in $\C_n$, if we write $\Arr^{n+1}(X)$ for the category of $(n+1)$-fold morphisms with codomain $X$, then there is an induced pair of adjoint functors:
\[
	e_*\colon \Arr^{n+1}(E) \to \Arr^{n+1}(B)
\]
left adjoint to
\[
	e^*\colon\Arr^{n+1}(B) \to \Arr^{n+1}(E),
\]
where $e_*(k \colon X \to E) \coloneq e\comp k$, and $e^*(h\colon X \to B)$ is given by the pullback of $h$ along~$e$. This adjunction also restricts to the categories of $(n+1)$-fold extensions above $E$ and $B$: the similarly defined adjunction
\[
	e_*|\colon \Ext^{n+1}(E) \to \Ext^{n+1}(B)\qquad\dashv\qquad e^*|\colon\Ext^{n+1}(B) \to \Ext^{n+1}(E).
\]
We say that $e$ is \defn{of effective descent} if $e_*$ is monadic, and $e$ is of effective $\E_n$-descent if $e_*|$ is monadic (see~\cite{JanTho1997,JanSoTho2004}).

\begin{remark}\label{RemarkPullbacksWheneverNeeded}
	The only requirement on a category $\C$, for the descent theory developed in~\cite{JanTho1994} to hold, is the existence of pullbacks. This requirement can be weakened to the existence of pullbacks ``whenever needed''. As we are working with higher extensions only, we can indeed compute pullbacks whenever needed (for the developments of~\cite{JanTho1994} which we use), even though we haven't proved the existence of arbitrary pullbacks in the category $\C_n$.
\end{remark}

In this section we show (recall) that higher extensions are of effective descent or effective $\E_n$-descent in $\C_n$. From~\cite[Lemma 3.2]{EGoeVdL} and the above, we have what can be understood as \emph{local $\E_n$--Barr exactness}:

\begin{proposition}\label{PropositionBarrKock}
	Assuming that $\C_0$ is Barr exact, and given a commutative square of $n$-fold extensions $\sigma = (\sigma_{\ttop},\sigma_{\pperp})$, together with the horizontal kernel pairs in $\C_{n-1}$ and the factorisation $f$ between them as in
	\begin{equation}\label{DiagramBarrKock}
		\vcenter{\xymatrix @!0@=5em{
		\Eq(\sigma_{\ttop}) \ar[d]_-{f} \ar@<.5ex>[r]^-{d_{\ttop}} \ar@<-.5ex>[r]_-{c_{\ttop}} & E_{\ttop} \ar[r]^-{\sigma_{\ttop}} \ar[d]^-{f_E} & B_{\ttop} \ar[d]^-{f_B}\\
		\Eq(\sigma_{\pperp}) \ar@<.5ex>[r]^-{d_{\pperp}} \ar@<-.5ex>[r]_-{c_{\pperp}} & E_{\pperp} \ar[r]_-{\sigma_{\pperp}} & \ B_{\pperp}
		}} \qquad \qquad \vcenter{\xymatrix@!0@=3em{
		& \Eq(\sigma_{\ttop}) \pullback \ar[rr]^-{d_{\ttop}} \ar[dd]|(.3){c_{\ttop}}|-{\hole} \ar[ld]_-{f} && E_{\ttop} \ar[dd]^-{\sigma_{\ttop}} \ar[ld]|-{f_E} \\
		\Eq(\sigma_{\pperp}) \pullback \ar[rr]^(.75){d_{\pperp}} \ar[dd]_-{c_{\pperp}} && E_{\pperp} \ar[dd]^(.25){\sigma_{\pperp}} \\
		& E_{\ttop} \ar[ld]|-{f_E} \ar[rr]^(.25){\sigma_{\ttop}}|-{\hole} && B_{\ttop} \ar[ld]^-{f_B} \\
		E_{\pperp} \ar[rr]_-{\sigma_{\pperp}} && B_{\pperp}}}
	\end{equation}
	the morphism $\sigma$ is an $(n+1)$-fold extension if and only if any (hence both) of the two left-hand (commutative) squares (i.e., $(d_{\ttop},d_{\pperp})$ or $(c_{\ttop},c_{\pperp})$) is an $(n+1)$-fold extension. If so, then
	\begin{itemize}
		\item $\sigma$ is the coequaliser, in $\C_n$, of the parallel pair $(d_{\ttop},d_{\pperp})$, $(c_{\ttop},c_{\pperp})\colon f \rightrightarrows f_E$, which is in turn the kernel pair of $\sigma$ in $\C_n$;
		\item the equivalence relation $\Eq(\sigma)$ in $\C_n$ is \emph{stably effective} in the sense of~\cite{JanTho1994};
		\item in particular, $(n+1)$-fold extensions are the coequalisers of their kernel pairs (computed point-wise in $\C_0$), i.e., $(n+1)$-fold extensions are regular epimorphisms in $\C_n$;
		\item all of these constructions (e.g., kernel pairs, coequalisers, etc.) are pointwise in $\C_0$: the exact fork in Diagram~\ref{DiagramBarrKock} consists of $2^{n}$ exact forks in~$\C_0$.
	\end{itemize}
\end{proposition}
\begin{proof}
	The first part is a direct consequence of~\cite[Lemma 3.2]{EGoeVdL}, whose proof remains valid when using our weaker version of right cancellation, Lemma~\ref{LemmaWeakRightCancellation} above. Since the component-wise coequaliser $\sigma$ of $(d_{\ttop},d_{\pperp})$, $(c_{\ttop},c_{\pperp})\colon f \rightrightarrows f_E$ is in particular a pushout square~\cite[Lemma 1.2]{Bou2003}, it coincides with the coequaliser in $\C_n$. Then $((d_{\ttop},d_{\pperp}),(c_{\ttop},c_{\pperp}))$ is the kernel pair of $\sigma$ since pullbacks along $(n+1)$-fold extensions are computed component-wise. Observe moreover that the parallel pair $(d_{\ttop},d_{\pperp})$, $(c_{\ttop},c_{\pperp})$ consists point-wise in a collection of $2^{n}$ equivalence relations in $\C$, linked by the $(n+1)$-fold extension $(f,f_E)$. Moreover, the $2^{n}$ morphisms which determine $\sigma$ are the coequalisers in $\C_0$ of these $2^{n}$ equivalence relations. Hence this kernel pair $((d_{\ttop},d_{\pperp}),(c_{\ttop},c_{\pperp}))$ is stably effective, because everything is computed point-wise, and all these point-wise equivalence relations are stably effective in the base category~$\C_0$, which is Barr exact.
\end{proof}

In Proposition~\ref{PropositionBarrKock} we are given the whole of Diagram~\ref{DiagramBarrKock}. By working pointwise in~$\C_0$, it is possible to recover the right-hand side of the diagram from the left-hand side (provided the horizontal pair of morphisms consists of $(n+1)$-fold extensions).

\begin{corollary}\label{CorollaryCoequalisersOfPointwiseEquivalenceRelations}
	Given an equivalence relation $\Eq(\sigma)$ in $\C_n$, such that its projections are $(n+1)$-fold extensions and it is point-wise a collection of equivalence relations in~$\C_0$, then its coequaliser exists in $\C_n$, it is an $(n+1)$-fold extension and it is computed point-wise via the coequalisers of each point-wise equivalence relation in~$\C_0$.\noproof
\end{corollary}

\begin{remark}\label{RemarkConsequencesOfBarrKock}
	Since $(n+1)$-fold extensions are regular epimorphisms in $\C_n$ (by Proposition~\ref{PropositionBarrKock}), an $(n+1)$-fold extension which is a monomorphism in $\C_n$ is an isomorphism. Moreover, by Proposition~\ref{PropositionBarrKock} and~\cite[Lemma 1.2]{Bou2003}, an $(n+1)$-fold extension is a pushout square of $n$-fold extensions in $\C_{n-1}$. Hence since pushouts preserve isomorphisms, looking at $\sigma$ in Diagram~\ref{DiagramBarrKock}, we see that if $\sigma_{\ttop}$ is an isomorphism, then so is $\sigma_{\pperp}$.
\end{remark}

Using the previous observation, we obtain:

\begin{lemma}\label{LemmaCompositeOfDoubleExtensionsIsPullback}
	If a composite of two $(n+1)$-fold extensions $\beta \comp \alpha$ yields a pullback square in $\C_{n-1}$, then both $\alpha$ and $\beta$ are pullback squares in $\C_{n-1}$.
\end{lemma}
\begin{proof}
	Remember that if the composite of two squares $\beta \comp \alpha$ is a pullback square, then $\alpha$ has to be a pullback---without any other assumptions on $\C_{n-1}$, $\alpha$ or $\beta$. In the diagram in Figure~\ref{DiagramCompositeOfDExtension}, all morphisms are $n$-fold extensions, since both $\alpha$ and $\beta$ are assumed to be $(n+1)$-fold extensions. Moreover, the square $q_2\comp\phi_3=\phi_2\comp\pi_2$ is a pullback square of regular epimorphisms in $\C_{n-1}$ and thus it is also a pushout~\cite{Carboni-Kelly-Pedicchio}. Finally, both $\phi_3$ and $\phi_1 \comp \phi_3$ are isomorphisms, hence $\phi_1$ is an isomorphism; and since pushouts preserve isomorphisms, also $\phi_2$ is an isomorphism---see Remark~\ref{RemarkConsequencesOfBarrKock}.
\end{proof}

\begin{lemma}\label{LemmaBarrKockForPullbacks}
	Under the hypotheses and notation of Proposition~\ref{PropositionBarrKock}, we see that the right-hand square $\sigma$ is a pullback square of $n$-fold extensions if and only if any (hence both) of the two left-hand (commutative) squares ($(d_{\ttop},d_{\pperp})$ or $(c_{\ttop},c_{\pperp})$) is a pullback square of $n$-fold extensions.
\end{lemma}
\begin{proof}
	Using Proposition~\ref{PropositionBarrKock}, in any case, we may assume that both projections $(d_{\ttop},d_{\pperp})$ and $(c_{\ttop},c_{\pperp})$ as well as $\sigma$ are $(n+1)$-fold extensions. Then looking at the cube on the right-hand side of Diagram~\ref{DiagramBarrKock}, if the right-hand face of the cube is a pullback square, then the composite of the ``back and right'' faces is a pullback, and thus the composite of the ``left and front'' faces is a pullback, and thus also the left-hand face is a pullback. Conversely, if the left-hand face is a pullback, the ``back and right'' composite is a pullback, and by Lemma~\ref{LemmaCompositeOfDoubleExtensionsIsPullback}, also the right-hand face is a pullback.
\end{proof}

Bearing in mind Remark~\ref{RemarkPullbacksWheneverNeeded}, as in~\cite[Remark~4.7]{EGVdL} we easily obtain:
\begin{proposition}\label{PropositionDoubleExtensionsEffectiveDecent}
	Any $(n+1)$-fold extension is of (global) effective descent in~$\C_n$.
\end{proposition}
\begin{proof}
	Note that the result holds for $n=1$. By induction, assume that it holds for all $n < m$ for some fixed $m >2$. Let $\sigma\colon f_E \to f_B$ be an $(m+1)$-fold extension. The monadicity of $\sigma_*$ in each component $\sigma^{\ttop}_*$ and $\sigma^{\pperp}_*$ easily yields the monadicity of $\sigma_*$ itself. For instance, we may use the characterisation in terms of \emph{discrete fibrations}\footnote{These are not to be confused with the discrete fibrations between extensions which form the subject of Section~\ref{Section Discrete Fibrations Double}.} of equivalence relations from Theorem~3.7 in~\cite{JanSoTho2004}.

	Consider a discrete fibration of equivalence relations $f_R \colon R_{\ttop} \to R_{\pperp}$ above the kernel pair $\Eq(\sigma)\colon\Eq(\sigma_{\ttop}) \to \Eq(\sigma_{\pperp})$ of $\sigma = (\sigma_{\ttop},\sigma_{\pperp})$ as in the commutative diagram of plain arrows below---see also Diagram~\ref{DiagramBarrKock}. This means that the squares of first and second projections are pullbacks. Then observe that $f_R$ is computed component-wise and consists in a comparison map between a pair of discrete fibrations of equivalence relations $R_{\ttop}$ and $R_{\pperp}$, above the pair of kernel pairs $\Eq(\sigma_{\ttop})$ and $\Eq(\sigma_{\pperp})$ with comparison map $\Eq(\sigma) \colon\Eq(\sigma_{\ttop}) \to \Eq(\sigma_{\pperp})$. By induction, it is computed point-wise, and consists in a comparison $(n+2)$-fold extension between a collection of $2^{n+1}$ discrete fibrations above the kernel pairs of the point-wise collection of $2^{n+1}$ morphisms defining $\sigma$. The projections of the equivalence relation~$f_R$ are also $(n+1)$-fold extensions, as pullbacks of the projections of $\Eq(\sigma)$---which are themselves $(n+1)$-fold extensions by Proposition~\ref{PropositionBarrKock}. Hence we may build the square on the right below, first by taking the coequaliser $\gamma = (\gamma_{\ttop},\gamma_{\pperp})$ of $f_R$ using Corollary~\ref{CorollaryCoequalisersOfPointwiseEquivalenceRelations}. The factorisation $(\bar{\beta}_{\ttop}, \bar{\beta}_{\pperp})$ is then obtained by the universal property of $\gamma$.
	\[
		\xymatrix @!0@=4em {
		f_R \opsplitpullback \ar[d]_{\hat{\beta}} \ar@<.5ex>[r] \ar@<-.5ex>[r] & f_C \ar@{{}{--}{>}}[r]^{(\gamma_{\ttop},\gamma_{\pperp})} \ar[d]_{\beta} & f_D \ar@{{}{..}{>}}[d]^{(\bar{\beta}_{\ttop},\bar{\beta}_{\pperp})}\\
		\Eq(\sigma) \ar@<.5ex>[r] \ar@<-.5ex>[r] & f_E \ar[r]_{\sigma} & \ f_B
		}
	\]
	By Proposition~\ref{PropositionBarrKock} and Corollary~\ref{CorollaryCoequalisersOfPointwiseEquivalenceRelations}, $f_R$ is the kernel pair of $\gamma$, which is an $(n+1)$-fold extension. Furthermore, the square on the right is a pullback square by Lemma~\ref{LemmaBarrKockForPullbacks}. Finally, $\gamma$ is pullback-stable as a coequaliser, since everything is computed point-wise: its kernel pair $f_R$ is stably effective by Proposition~\ref{PropositionBarrKock}.
\end{proof}

What exactly we need in our context is not \emph{global} effective descent but effective $\E_n$-decent. This derives from Proposition~\ref{PropositionDoubleExtensionsEffectiveDecent} because of Lemma~\ref{LemmaPullbacksReflectDExtensions}, as explained in~\cite[Section 2.7]{JanTho1994}.

\begin{corollary}
	Any $(n+1)$-fold extension is of effective $\E_{n+1}$-descent in $\C_{n}$.\noproof
\end{corollary}

\section{Some properties of strongly Birkhoff Galois theories}\label{Intermediate results}

In this section we collect some intermediate results which, as explained in Subsection~\ref{Subsection Tower}, are meant to simplify the process of showing that, in certain specific situations, there is a tower of strongly Birkhoff Galois structures above a given strongly Birkhoff Galois structure.

\subsection{The Birkhoff condition in higher dimensions}\label{SubsectionClosureByQuotientsAndBirkhoff}
First we work towards Proposition~\ref{PropositionBirkhoffIffClosedUnderQuotient}, which is a higher-dimensional version of the equivalence between the Birkhoff condition (as in Subsection~\ref{Subsection Categorical Galois Theory}) and closedness under quotients of the class of coverings.

\begin{lemma}\label{LemmaTrivDoublePushoutsReduceToPushouts}
	For some $n\geq 1$, we consider in $\C_n$ a commutative diagram of $(n+1)$-fold extensions
	\begin{equation}\label{Diagram z}
		\vcenter{\xymatrix@!0@=4em{
		A \ar[r]^-{z_A} \ar[d]_-{\phi} & \bar{A} \ar[d]^-{\bar{\phi}}\\
		B \ar[r]_-{z_B} & \bar{B}.}}
	\end{equation}
	We require of $z_B\colon B \to \bar{B}$ that all of its $\C_0$-components---see Subsection~\ref{SubsectionHigherMorphisms}---are identity morphisms, with the eventual exception of the initial component denoted $s_B\colon{B_{\meet}\to \bar B_{\meet}}$.

	This square~\eqref{Diagram z} is a pushout in $\C_n$ if and only if the induced square in Figure~\ref{DiagramTopPushout}
	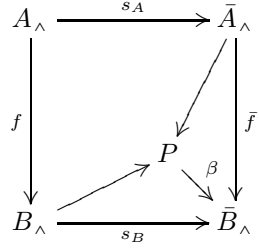
\begin{figure}
		$\xymatrix @!0@=2.5em{
			A_{\meet} \ar[rrr]^-{s_A} \ar[ddd]_{f} & & & \bar{A}_{\meet} \ar[ddl] \ar[ddd]^{\bar{f}} \\
			\\
			& & P \ar[rd]^-{\beta} & \\
			B_{\meet} \ar[rru] \ar[rrr]_-{s_B} & & & \bar{B}_{\meet}
			}$
		\caption{Induced pushout of initial objects}\label{DiagramTopPushout}
	\end{figure}
	between the initial objects of $A$, $\bar{A}$, $B$ and $\bar{B}$ is a pushout in $\C_0$. In other words, the comparison $\beta$ from the pushout $P\coloneq B_\meet +_{A_\meet}\bar{A}_\meet$ to $\bar{B}_\meet$ is an isomorphism.
\end{lemma}

Note that, in the statement above, if similarly all of the $\C_0$-components of $z_A$ are identity morphisms in $\C_0$, but for the initial component, then the square~\eqref{Diagram z} is an $(n+2)$-fold extension by Lemma~\ref{LemmaHigherExtensionsCalculus}.

\begin{proof}[Proof of Lemma~\ref{LemmaTrivDoublePushoutsReduceToPushouts}]
	If $\beta$ in Figure~\ref{DiagramTopPushout} is an isomorphism, then testing for the universal property of Diagram~\ref{Diagram z}, we consider the following commutative diagram of plain arrows.
	\begin{equation*}\label{DiagramTestPushout}
		\vcenter
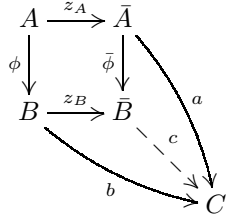
{\xymatrix@!0@=3.5em{
		A \ar[r]^-{z_A} \ar[d]_-{\phi} & \bar{A} \ar@/^/[ddr]^-{a} \ar[d]_-{\bar{\phi}} & \\
		B \ar[r]^-{z_B} \ar@/_/[drr]_-{b} & \bar{B} \ar@{{}{--}{>}}[dr]^(0.4){c} & \\
		& & C
		}}
	\end{equation*}
	The morphism $c$ consists of the $\C_0$-components of $b$ in all components but the initial one. There it is the morphism induced by the pushout $\bar{B}_{\meet}$ from Figure~\ref{DiagramTopPushout}. This is of course only possible because all the $\C_0$-components of $z_B$, but the initial one, are identity morphisms. To test the remaining commutativity conditions that make this collection of $2^n$ morphisms in $\C_0$ into a morphism in $\C_n$, it suffices to pre-compose the underlying $(n+1)$-cubical diagram of $c$ with the epimorphism $s_B$. To show that $c\comp \bar{\phi} = a$, it suffices to pre-compose in $\C_n$ with the epimorphism $z_A$.

	Conversely, if Diagram~\ref{Diagram z} is a pushout, then we want to show that $\beta$ in Figure~\ref{DiagramTopPushout} is an isomorphism. We consider the $n$-cubical diagram $\bar{P}$, obtained from~$\bar{B}$ by pre-composing all its morphisms of domain~$\bar{B}_{\meet}$ with $\beta$. The initial object of $\bar{P}$ is then $P$. By~\cite[Proposition 1.16]{EGoeVdL}, $\bar{P}$ is still an $n$-fold extension. Similarly, we have induced morphisms $z'_B \colon B \to \bar{P}$ and $\bar{\phi}' \colon \bar{A} \to \bar{P}$ defining a commutative square $\bar{\phi}'\comp z_A = z'_B\comp \phi$ in $\C_n$. Hence the universal property of the pushout~\eqref{Diagram z} induces a morphism $\bar{B} \to \bar{P}$ whose initial component is an inverse for $\beta$.
\end{proof}

In order to state the next result, we assume that we are given a partially defined tower of Galois structures as in Subsection~\ref{Subsection Tower}. For some fixed $n\geq 1$, for each $k\leq n$ we have a Galois structure $\Gamma_k$ which is strongly $\E_{k+1}$-Birkhoff, where $\C_k$ is the category of $k$-fold extensions over $\C_0$ and $\B_k$ is the category $\Cov(\C_{k-1})$ of $\Gamma_{k-1}$-coverings. Note that by~\cite[Corollary~5.2]{Im-Kelly}, at each level $F_k\dashv \I_k$, for any object $A$ in $\C_k$, all of the $\C_0$-components of the unit morphism $\eta^k_A\colon A\to \I_kF_k(A)$ are identity morphisms---with the possible exception of the initial component. The following result is a generalisation of Proposition~3.1 in~\cite{JK1994} which covers the base case $n=0$.

\begin{proposition}\label{PropositionBirkhoffIffClosedUnderQuotient}
	Suppose that $\B_{n+1}$ is reflective in $\C_{n+1}$ with reflector $F_{n+1}$. The adjunction $F_{n+1}\dashv\I_{n+1}$ is $\E_{n+2}$-Birkhoff if and only if $\B_{n+1}$ is closed under quotients along $(n+2)$-fold extensions in $\C_{n+1}$.
\end{proposition}
\begin{proof}
	$(\Rightarrow)$ is a direct consequence of the fact that isomorphisms are stable under pushouts. For the proof of $(\Leftarrow)$, we use that the pushout $P$ of $f$ and $\eta_A^n$ in the diagram in Figure~\ref{DiagramUnitPushout}
	\begin{figure}
		$\xymatrix @!0@=2.5em{
			A \ar[rrr]^-{\eta^n_A} \ar[ddd]_{f} & & & F_n(A) \ar@{-->}[ddl]_-{\bar{f}} \ar[ddd]^-{F_n(f)} \\
			\\
			& & P \ar@{.>}[rd]^-{\beta} & \\
			B \ar@{-->}[rru]^-{\bar\eta^n_A} \ar[rrr]_-{\eta^n_B} & & & F_n(B)
			}$
		\caption{Reflection square and induced pushout}\label{DiagramUnitPushout}
	\end{figure}
	exists and is constructed as in the second part of the proof of Lemma~\ref{LemmaTrivDoublePushoutsReduceToPushouts}. All arrows in the diagram are $(n+1)$-fold extensions. Hence, in particular, $P$ is the quotient of $F_n(A)$ along the $(n+1)$-fold extension $\bar f$, so that it is a covering by our assumption. The universal property of $F_n(B)$ now provides an inverse for $\beta$.
\end{proof}

Now we prove that the second condition in the statement of Proposition~\ref{PropositionBirkhoffIffClosedUnderQuotient} always holds, so that the adjunction is automatically Birkhoff: see Corollary~\ref{BirkhoffConclusion} below.

\subsection{Closedness under quotients along double extensions}\label{SubsectionStronglyBirkhoffAndCoveringsClosedAlongExtensions}

Let $\Gamma$ be a Galois structure as in Convention~\ref{ConventionGaloisStructure} and let $\X$ satisfy the strong $\E$-Birkhoff condition. As before, we have the inclusion $\Ii \colon \Cov(\Y) \to \Ext(\Y)$, and we know that the pullback, computed in $\Y$, of a covering along an extension yields a covering. Coverings are thus preserved by pullbacks along extensions.

Likewise, coverings are also reflected by pullbacks along extensions. Indeed, given a pullback square of extensions as on the left in Figure~\ref{Coverings Reflected}, if $f_A$ is a covering, then by assumption there is $p$ such that the projection $\pi_E$ of the pullback on the right in Figure~\ref{Coverings Reflected} is a trivial covering, and this shows that $f_B$ is split by the composite~$\phi_0\comp p$.
\begin{figure}
	$\xymatrix@!0@=1.3em{
		E\times_{A_0} A_1 \pullback \ar[rrrr]^-{\pi_{A_1}} \ar[dddd]_-{\pi_E} & & & & A_1 \pullback \ar[rrrr]^-{\phi_1} \ar[dddd]_-{f_A} & & & & B_1 \ar[dddd]^-{f_B} \\
		\\
		\\
		\\
		E \ar[rrrr]_-{p} & & & & A_0 \ar[rrrr]_-{\phi_0} & & & & B_0
		}$
	\caption{Coverings reflected by pullbacks along extensions}\label{Coverings Reflected}
\end{figure}

This result implies that $\Cov(\Y)$ is closed under quotients along $\E$--double extensions in $\Ext(\Y)$:

\begin{lemma}\label{FactorisingTrivialExtensions}
	If $\X$ is strongly $\mathcal{E}$-Birkhoff in $\Y$, then given extensions $t$ and~$s$ such that $f_A=s\comp t$ is a trivial covering, the extensions $t$ and $s$ are trivial coverings as well.
\end{lemma}
\begin{proof}
	This is an immediate consequence of Lemma~\ref{LemmaCompositeOfDoubleExtensionsIsPullback}, applied to the diagram in Figure~\ref{Figure Composing Unit Squares}
	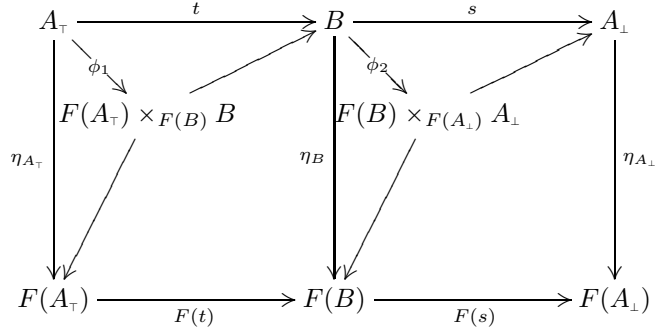
\begin{figure}
		$\xymatrix@!0@=3.5em{
			A_{\ttop} \ar[rrr]^-{t} \ar[rd]|-{\phi_1} \ar[ddd]_{\eta_{A_{\ttop}}} & & & B \ar[rd]|-{\phi_2} \ar[ddd]_-{\eta_{B}} \ar[rrr]^-{s} & & & A_{\pperp} \ar[ddd]^-{\eta_{A_{\pperp}}}\\
			& F(A_{\ttop}) \times_{F(B)}B \ar[rru] \ar[ldd]& & & F(B) \times_{F(A_{\pperp})}{A_{\pperp}} \ar[rru] \ar[ldd]\\
			\\
			F(A_{\ttop}) \ar[rrr]_{F(t)} & & & F(B) \ar[rrr]_{F(s)} & & & F(A_{\pperp})
			}$
		\caption{Composing unit squares}\label{Figure Composing Unit Squares}
	\end{figure}
	where, by the strong $\E$-Birkhoff condition, both squares are double extensions.
\end{proof}

Pullback-stability of extensions implies:

\begin{lemma}\label{FactorisingCentralExtensions}
	If $\X$ is strongly $\mathcal{E}$-Birkhoff in $\Y$, then given extensions $t$ and~$s$ such that $f_A=s\comp t$ is a covering, the extensions $t$ and $s$ are coverings as well.\noproof
\end{lemma}

The fact that in this context, coverings are reflected by pullbacks along extensions, now gives us:

\begin{proposition}\label{CentralExtensionsAreStableAlongDoubleExtensions}
	If $\X$ is strongly $\mathcal{E}$-Birkhoff in $\Y$, then coverings are stable under quotients along double extensions.\noproof
\end{proposition}

\begin{corollary}\label{BirkhoffConclusion}
	Under the hypotheses of Proposition~\ref{PropositionBirkhoffIffClosedUnderQuotient}, the Galois structure $\Gamma_{n+1}$ is $\E_{n+2}$-Birkhoff.\noproof
\end{corollary}

\subsection{Towers of Galois structures in algebra}\label{SubsectionAlgebraicExamples}

In order for the above to yield a tower of strongly Birkhoff Galois structures in concrete examples, two ingredients are still missing: first of all, the reflectiveness of coverings at each step, and secondly, a condition that ensures that reflection squares are extensions rather than just pushouts. Here we consider a class of examples where those two conditions are met.

In~\cite{Ever2014}, T.~Everaert develops a theory of towers of Galois structures in a more restrictive context where all categories involved share the Mal'tsev property. In that context, any base Galois structure which is Birkhoff automatically gives rise to an appropriate tower. In the article, an explicit construction for each reflector~$F_n$ is given, which relies on the fact that in that context, coverings are always split by themselves---in other words, they coincide with normal coverings. In the present, more general context, the construction of a reflection at each level requires a systematic method for obtaining splittings for coverings.

Assuming for instance that $\B_0$ is a subvariety of a variety of algebras $\C_0$, let us write the free/forgetful adjunction to $\Set$ as $F_a\dashv \U$. Then every covering $c\colon A \to B$ is split by the canonical $\E_1$-projective presentation ${\epsilon^a_{B} \colon F_a \U (B) \to B}$ of its codomain $B$, which provides the needed ``systematic splitting''. Here, the Galois structure $\Gamma_0$ is always $\E_1$-Birkhoff: reflection squares are pushouts. (Note that this universal-algebraic situation is precisely where originally the term ``Birkhoff'' came from.) The \emph{strong} Birkhoff condition, which says that the reflection squares are extensions, then (by Proposition~5.4 in~\cite{Carboni-Kelly-Pedicchio}) amounts to the condition that the kernel pair of each unit morphism $\eta_A$ permutes (in the sense of composition of relations) with any congruence on the object~$A$.

The idea of the following construction (which is a variation on the one in~\cite{Ever2014,DuEvMo2017}) generalises to arbitrary levels in the process of building a tower of Galois structures. Note that we are not claiming that the result is new---see for instance~\cite{JK:Reflectiveness}; rather, we wish to sketch a concrete construction of an adjoint. In fact, we shall only give a detailed description of this construction at the base level; assuming that we have $\Gamma_n$ and are constructing the reflector $F_{n+1}$, the argument is essentially the same, mainly because for each object~$A$ of $\C_n$, all the $\C_0$-components of $\eta_A^n$ are identities, except for the initial one.

\begin{proposition}\label{PropReflctOfCentralExtensions}
	Adding that $\C_0$ is a category with enough projectives to the hypotheses of Proposition~\ref{PropositionBirkhoffIffClosedUnderQuotient} (and also considering the case $n=0$), we find that the category $\B_{n+1} = \Cov(\C_n)$ is reflective in $\C_{n+1}=  \Ext(\C_n)$.
\end{proposition}
\begin{proof}%[Proof for $n=0$]
	As usual, it suffices that we build the unit of the adjunction, and show its universal property. A reflector is then uniquely determined. As explained in~\cite{DuEvMo2017}, coverings are pullback-stable, and thus by Proposition~5.8 in~\cite{Im-Kelly}, we may work locally above each object $B$ and build a left adjoint to the inclusion of the category $\Cov(B)$ of coverings above $B$ into the category $\Ext(B)$ of extensions above~$B$.

	We give the details for $n=0$. We fix $B$ together with a projective presentation $p\colon E\to B$ of $B$. We consider an extension $f\colon A\to B$ and the induced pullback
	\[
		\xymatrix@!0@=5em{
		\Eq(\pi_A) \opsplitpullback \ar@{-->}[d]_{\hat{f}} \ar@<.5ex>[r] \ar@<-.5ex>[r] & E \times_{B} A \pullback \ar[r]^-{\pi_A} \ar[d]_{\pi_E} & A \ar[d]^{f}\\
		\Eq(p) \ar@<.5ex>[r] \ar@<-.5ex>[r] & E \ar[r]_{p} & B
		}
	\]
	where we take the kernel pairs of $\pi_A$ and $p$ and construct the factorisation $\hat{f}$. Note that on higher levels, projective presentations can be obtained from those at the base level, while the constructions in the diagram above are the same, thanks to the properties of higher extensions discussed in Section~\ref{Section Calculus}.

	The strongly Birkhoff Galois structure $\Gamma_0$ yields a trivialisation functor which we apply to the double equivalence relation on the left, which yields a pair of arrows $(\pi_1, \pi_2)$ as in Figure~\ref{Construction of the reflector}.
	\begin{figure}
		$\xymatrix@!0@=3.5em{
			& \Eq(\pi_A) \opsplitpullback\ar[dl]_{\tau} \ar[ddd]|(.32){\hole}|(.35){\hole}^(.6){\hat{f}} \ar@<.5ex>[rrr] \ar@<-.5ex>[rrr] & & & E \times_{B} A \pullback \ar[dl]_{\tau} \ar[rrr]^{\pi_A} \ar[ddd]|(.33){\hole}^(.6){\pi_E} & & & A \ar@{{}{..}{>}}[dl]_{\eta_{A}^1} \ar[ddd]^{f}\\
			\cdot \ar[rdd]_-{\Triv(\hat{f})} \ar@<.5ex>[rrr]^(.66){\pi_1} \ar@<-.5ex>[rrr]_(.66){\pi_2} & & & \cdot \ar[rdd]_-{\Triv(\pi_E)} \ar[rrr]^(.66){r} & & & C_f \ar@{{}{..}{>}}[rdd]_-{F_1(f)} \\
			\\
			& \Eq(p) \ar@<.5ex>[rrr] \ar@<-.5ex>[rrr] & & & E \ar[rrr]_{p} & & & B
			}$
		\caption{Construction of the reflector $F_1$}\label{Construction of the reflector}
	\end{figure}
	As usual, the adjunction unit $\tau$ is only non-trivial in its initial component. We may then take the coequaliser $r$ of $\pi_1$ and $\pi_2$, which yields the dotted factorisation of~$f$ in the diagram above. On higher levels, the construction of this coequaliser~$r$ is still possible thanks to the shape of the unit morphisms involved. All the thus obtained morphisms will indeed be higher extensions. Lemma~\ref{LemmaWeakRightCancellation} tells us that the front right commutative square $F_1(f)\comp r=p\comp \Triv(\pi_E)$ is a double extension. By closedness of coverings under quotients along double extensions, $F_1(f)$ is a covering. Moreover, the top square on the right is a pushout since the~$\tau$ are epimorphisms. On higher levels we here use Lemma~\ref{LemmaTrivDoublePushoutsReduceToPushouts}. (We may further notice here that if instead of being a projective presentation, the morphism $p$ is a weakly universal normal covering, then all of the squares in the factorisation are pullbacks.)

	We define $\eta^1_f\colon f\to F_1(f)$ to be the couple $(\eta_A^1,1_B)$ and prove that it satisfies the universal property of a unit. As mentioned above, we only need to check this in $\Ext(B)$. We consider another factorisation $c\comp t =f$ as in Figure~\ref{Figure Universal Property}, where $c$ is a covering. By definition, there is an extension $q\colon Q\to B$ such that the pullback $c'$ of $c$ along $q$ is a trivial covering. Since $p$ is projective, there is a factorisation $s$ such that $p=q\comp s$. Then, the pullback $c''$ of $c$ along $p$ is also the pullback of the trivial covering $c'$ along $s$. Hence $c''$ is a trivial covering. We have the diagram in Figure~\ref{Figure Universal Property}
	\begin{figure}
		$\xymatrix@!0@=4em {
			& E \times_{B} A \pullback \ar@{{}{--}{>}}[rdd]^(.3){h} \ar[dl]_{\tau} \ar[rrr]^{\pi_A} \ar[ddd]|(.33){\hole}|(.75){\pi_E} & & & A \ar[dl]_{\eta_{A}^1} \ar[ddd]^(.75){f} \ar[rdd]^{t} & \\
			\cdot \ar@{{}{..}{>}}[rrd]|(.35){h'} \ar[rdd]_-{\Triv(\pi_E)} \ar[rrr]^(.75){r} & & & C_f \ar[rdd]|<<<<<<<{F_1(f)} \ar@{.>}[rrd]^(.66)x & & \\
			& & E\times_BT \skewpullback\ar[ld]^{c''} \ar[rrr]|(.5){\hole}|(.66){\hole} & & & T \ar[ld]^{c} \\
			& E \ar[rrr]_{p} & & & B
			}$
		\caption{Checking the universal property of $F_1(f)$}\label{Figure Universal Property}
	\end{figure}
	where $h=1_E\times_Bt$ is induced by the universal property of the pullback $E\times_BT$ and $h'$ is induced by the universal property of the trivialisation $\tau$. All these morphisms are regular epimorphisms and all the displayed diagrams commute. Hence, since we showed that the top square is a pushout, we get a factorisation $x\colon C_f\to T$ as required. Uniqueness comes from the universal property of the pushout and the fact that all plain arrows in the diagram are regular epimorphisms. On higher levels, essentially the same argument is valid.
\end{proof}

\begin{remark}\label{Remark Quotient}
	One may wonder if the construction of $F_1$ using the coequaliser $r$ above should not yield the trivialisation functor, considering that $\Triv$ preserves coequalisers, as any left adjoint. Indeed, the long rectangle at the back in Figure~\ref{Construction of the reflector} is a coequaliser in $\Ext(\C)$, whose image by the trivialisation is thus a coequaliser in $\TCov(\C)$. However, coequalisers in $\TCov(\C)$ are different from coequalisers in~$\Ext(\C)$. Indeed, a quotient of a trivial extension in $\Ext(\C)$ is not a trivial extension in general, but it is a covering, whence this natural construction of $F_1$ using the quotient, along a universal splitting, of the trivialisation.
\end{remark}

Assuming again that $\B_0$ is a subvariety of a variety of algebras $\C_0$, the reflection is strongly Birkhoff if and only if the kernel pair of each unit morphism $\eta_A$ permutes with every congruence on the object $A$. (As mentioned before, this is a consequence of Proposition~5.4 in~\cite{Carboni-Kelly-Pedicchio}.) If so, then by the preceding results, $\Gamma_0$-coverings are reflective amongst (one-fold) extensions, and this reflection is automatically $\E_2$-Birkhoff. As for the base level, the only remaining condition to be checked in order to obtain a strongly $\E_2$-Birkhoff Galois structure $\Gamma_1$ is that a similar permutability condition holds for the unit morphisms at this new level. In the spirit of Lemma~\ref{LemmaTrivDoublePushoutsReduceToPushouts}, thanks to the shape of these unit morphisms, we only care about the initial component of their kernel pair (a congruence in $\C_0$). If this permutability condition holds in $\Gamma_1$, then we automatically obtain the $\E_3$-Birkhoff Galois structure $\Gamma_2$. Checking the appropriate permutability condition, we may see that $\Gamma_2$ is strongly $\E_3$-Birkhoff. We may then proceed by induction on $n$, checking the permutability condition at each new level $\Gamma_n$.

\subsection{Mal'tsev varieties}
In the context of a Mal'tsev variety $\B_0$, at each level of the tower the permutability condition comes for free. Here the process we describe is well studied in the literature: first in the context of semi-abelian categories, then in Mal'tsev categories~\cite{EGVdL,Ever2010,Ever2014}.

\subsection{Racks and quandles}\label{Racks Quandles}
Our motivating example is the \emph{connected components adjunction} in the variety of quandles (or, more generally, racks). The idea is that the coverings of~\cite{Eis2014} may be obtained from a Galois theoretic perspective as in~\cite{Eve2015}. Our aim here is to develop the corresponding higher-dimensional Galois theory. A~detailed study of the first two stages of this process (the corresponding $\Gamma_0$ and~$\Gamma_1$) may be found in~\cite{Ren2020,Ren2021}. The remaining levels of the tower will be discussed in~\cite{Ren2021c}. The reason why the permutability condition holds in this context is that each so-called \emph{centralisation relation} (which is the kernel pair of the initial component of a unit morphism at some level) can be viewed as an \emph{orbit congruence}---see~\cite{BLRY2010}. In order to prove this result at higher levels, it is convenient to know that the concept of a covering, and thus also these centralisation relations, satisfy a fundamental symmetry property. Describing this symmetry property, and showing that coverings satisfy it in a suitably general context, is the purpose of the remainder of this article.

\section{Discrete fibrations of double extensions}\label{Section Discrete Fibrations Double}

We now start our development of a theory of \emph{symmetric coverings}, which allows a non-inductive treatment of the concept of a covering. The basic idea is very simple and based on the concept of a \emph{discrete fibration}, which is what we call an $n$-fold extension which, considered as an $n$-cube, is a limit diagram---the easiest example being a pullback square of regular epimorphisms. Recall that, in a Galois structure $\Gamma = (\Y,\X,L,\I,\eta,\epsilon,\E)$, an extension $c$ as in the diagram
\[
	\vcenter{\xymatrix@!0@=4em{ X \ar[d]_-{x} & T \pullback \pullbackleft \ar[d]^-{t} \ar[l]_-{} \ar[r] & A \ar[d]^-{c} \\
	Y & E \ar[l]^-{} \ar[r]_-{e} & B
	}}
\]
is a $\Gamma$-\emph{covering} if and only if there is an extension $e$ such that the pullback of $c$ along~$e$ is a \emph{trivial $\Gamma$-covering} $t$, which means that $t$ is a pullback of a \emph{primitive $\Gamma$-covering}~$x$ (which is an extension in the given full reflective subcategory~$\X$ of the ground category $\Y$). In other words, $c$ is a covering when a span of discrete fibrations $x\ot t\to c$ exists in $\Y$, where $x$ lies in $\X$. This idea applies at arbitrary levels of a given tower of Galois structures; it suffices to replace the pullback squares with higher-order discrete fibrations as defined above. Our main result here is Theorem~\ref{Main Theorem}, which says that in the context of an appropriate tower of Galois structures, this approach does indeed characterise the inductively defined coverings of Galois theory---thus showing that the concept of a higher covering is symmetric: it does not depend on the way the $n$-cube is considered as an arrow between $(n-1)$-cubes.

In order to achieve Theorem~\ref{Main Theorem}, we first consider an $n$-cubical extension as a double extension in $\C_{n-2}$. Then we prove, in this section and in Section~\ref{Section Double Coverings}, that for such an $n$-fold extension, the property of being a $\Gamma_{n-1}$-covering admits a symmetrical interpretation in terms of discrete fibrations of order three in $\C_{n-2}$. In Section~\ref{Section Higher Coverings}, this is then used to obtain a symmetrical interpretation with respect to $\C_0$.

From now on, we assume that a tower of strongly Birkhoff Galois structures is given whose base level is $\Gamma_0$ with categories $(\C_0,\B_0)$ as before. For some fixed $n\geq 0$, we place ourselves at the level $\Gamma = \Gamma_n$ with $(\C,\B)=(\C_n, \B_n)$. An extension in $\C$ is an element of the class $\E = \E_{n+1}$. A \defn{double extension} is simply a $\E$-double extension. A $3$-cubical extension in $\C$ is the underlying diagram in $\C=\C_n$ of an $(n+3)$-fold extension---see Subsection~\ref{Subsection Ext and higher}. A (trivial/normal) covering is a (trivial/normal) $\Gamma$-covering. A (trivial/normal) double covering is a (trivial/normal) $\Gamma_{n+1}$-covering, etc.

\begin{definition}
	A $3$-cubical extension in $\C$ is a \defn{discrete fibration} when it is a limit cube: any cone of wavy arrows
	\[
		\vcenter
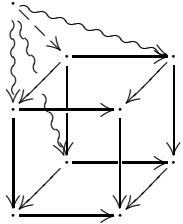
{\xymatrix@1@!0@=2em{\cdot \ar@{~>}[dd] \ar@{~>}[rrrd] \ar@{~>}[dddr]|-{\hole} \ar@{-->}[rd]\\
		& \cdot \ar@{->}[rr] \ar@{->}[ld] \ar@{->}[dd]|-{\hole} && \cdot \ar@{->}[ld] \ar@{->}[dd] \\
		\cdot \ar@{->}[rr] \ar@{->}[dd] && \cdot \ar@{->}[dd]\\
		& \cdot \ar@{->}[rr]|-{\hole} \ar@{->}[ld] && \cdot \ar@{->}[ld]\\
		\cdot \ar@{->}[rr] && \cdot}}
	\]
	on it induces a unique dashed comparison arrow.
\end{definition}

\begin{example}\label{Example DF Via Cone}
	Any $3$-cubical extension, obtained as a pullback of two double extensions in $\C$, is a discrete fibration. To see this, let us consider such a cube and a cone on it (the wavy arrows).
	\[
		\vcenter
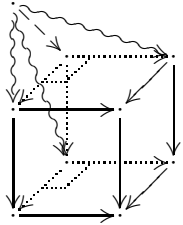
{\xymatrix@1@!0@=2em{\cdot \ar@{~>}[dd] \ar@{~>}[rrrd] \ar@{~>}[dddr] \ar@{-->}[rd]\\
		& \cdot \skewpullbackdots \ar@{.>}[rr] \ar@{.>}[ld] \ar@{.>}[dd] && \cdot \ar@{->}[ld] \ar@{->}[dd] \\
		\cdot \ar@{->}[rr] \ar@{->}[dd] && \cdot \ar@{->}[dd]\\
		& \cdot \skewpullbackdots \ar@{.>}[rr] \ar@{.>}[ld] && \cdot \ar@{->}[ld]\\
		\cdot \ar@{->}[rr] && \cdot}}
	\]
	The universal property of the top square induces a dashed arrow that makes the top part of the diagram commute. The rest commutes as well, because the bottom square is also a pullback.
\end{example}

We may consider a discrete fibration as an arrow between double extensions in three non-equivalent ways. By making such a choice, the cube acquires a domain~$\alpha$ and a codomain $\beta$. The entire structure is then called a \defn{discrete fibration from~$\alpha$ to~$\beta$}.

\begin{example}\label{Example DF arrow between pullbacks}
	The above example has the following alternate interpretation: any $3$-cubical extension (in $\C$) that happens to be an arrow between double extensions which are pullback squares is a discrete fibration. Actually, here the domain and codomain may themselves be seen as discrete fibrations (of a lower order): this is the viewpoint of Example~\ref{Example Pullback as Disc Fib}.
\end{example}

Alternatively, we may choose one of the three ways a three-cube may be seen as a square of two-cubes---see the diagram in the statement of Proposition~\ref{Prop DF as Pullback Up To a Pullback} below. This viewpoint allows us to sharpen Example~\ref{Example DF Via Cone} into the following general characterisation.

Note that in what follows, we use the symbol ``$\equiv$'' to indicate that we are considering the same diagram twice, depicted at two different levels of the tower of Galois structures. In Proposition~\ref{Prop DF as Pullback Up To a Pullback} for instance, the \emph{vertices} in the diagram on the left are $1$-cubical extensions in $\C$---objects in $\C_{n+1}$, at the $(n+1)$-st level of the tower induced by $(\C_0,\B_0)$---so may be seen as arrows in $\C$, as in the diagram on the right.

\begin{proposition}\label{Prop DF as Pullback Up To a Pullback}
	A three-cubical extension is a discrete fibration if and only if, when viewed as a square of double extensions---the outer quadrangle in the diagram on the left
	\[
		\vcenter
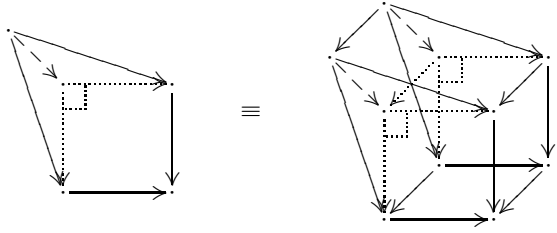
{\xymatrix@1@!0@=2em{\cdot \ar@{-->}[rd] \ar[drrr] \ar[dddr]\\
		&\cdot \pullbackdots \ar@{.>}[rr] \ar@{.>}[dd] && \cdot \ar@{->}[dd] \\\\
		&\cdot \ar@{->}[rr] && \cdot }}
		\qquad\equiv\qquad
		\vcenter{\xymatrix@1@!0@=2em{& \cdot \ar[ld] \ar@{-->}[rd] \ar[drrr] \ar[dddr]\\
		\cdot \ar@{-->}[rd] \ar[drrr] \ar[dddr] && \cdot \pullbackdots \ar@{.>}[rr] \ar@{.>}[ld] \ar@{.>}[dd] && \cdot \ar@{->}[ld] \ar@{->}[dd] \\
		&\cdot \pullbackdots \ar@{.>}[rr] \ar@{.>}[dd] && \cdot \ar@{->}[dd]\\
		&& \cdot \ar@{->}[rr] \ar@{->}[ld] && \cdot \ar@{->}[ld]\\
		&\cdot \ar@{->}[rr] && \cdot}}
	\]
	---the induced dashed comparison to the dotted pullback is not just a double extension, but a pullback square.
\end{proposition}
\begin{proof}
	It is easily seen that the universal properties those cubes satisfy coincide.
\end{proof}

Since a composite of two pullback squares is again a pullback, this gives us:

\begin{corollary}\label{Corollary DF compose}
	Discrete fibrations compose, in whichever way they are considered as arrows between double extensions. \noproof
\end{corollary}

\begin{example}\label{Visual example}
	Combining Corollary~\ref{Corollary DF compose} with Example~\ref{Example DF Via Cone}, we see that the horizontal composite of the cubes of extensions in the diagram
	\[
		\vcenter
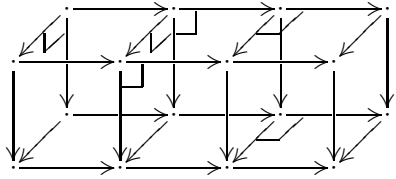
{\xymatrix@1@!0@=2em{ & \cdot \opskewpullback \ar@{->}[rr] \ar@{->}[ld] \ar@{->}[dd]|{\hole} && \cdot \pullback \opskewpullback \ar@{->}[rr] \ar@{->}[ld] \ar@{->}[dd]|{\hole} && \cdot \skewpullback \ar@{->}[rr] \ar@{->}[ld] \ar@{->}[dd]|{\hole} && \cdot \ar@{->}[ld] \ar@{->}[dd] \\
		\cdot \ar@{->}[rr] \ar@{->}[dd] && \cdot \pullback \ar@{->}[rr] \ar@{->}[dd] && \cdot \ar@{->}[rr] \ar@{->}[dd] && \cdot \ar@{->}[dd]\\
		& \cdot \ar@{->}[rr]|{\hole} \ar@{->}[ld] &&\cdot \ar@{->}[rr]|{\hole} \ar@{->}[ld] && \cdot \skewpullback \ar@{->}[rr]|{\hole} \ar@{->}[ld] && \cdot \ar@{->}[ld]\\
		\cdot \ar@{->}[rr] && \cdot \ar@{->}[rr] && \cdot \ar@{->}[rr] && \cdot}}
	\]
	is a discrete fibration.
\end{example}

\section{Double coverings}\label{Section Double Coverings}

\begin{definition}
	A double extension $\tau$ in $\C$ is called a \defn{symmetrically trivial double covering (with respect to $\Gamma$)} when a discrete fibration $\tau\to \beta$ exists whose codomain~$\beta$ lies in $\B$ (it is a primitive $\Gamma$-covering).
\end{definition}

\begin{lemma}\label{Lemma Trivial}
	Any symmetrically trivial double covering is a trivial double covering, in whichever way the square is considered as an arrow between arrows.
\end{lemma}
\begin{proof}
	Let $\tau\to \beta$ be a discrete fibration whose codomain $\beta$ lies in $\B$, the square on the left below.
	\[
		\vcenter
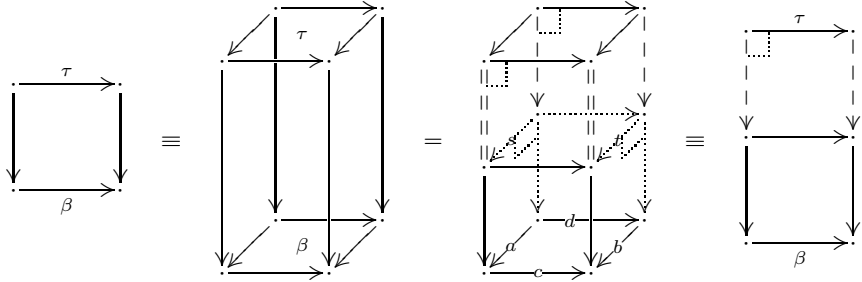
{\xymatrix@1@!0@=2em{\cdot \ar@{->}[rr]^-\tau \ar@{->}[dd] && \cdot \ar@{->}[dd] \\\\
		\cdot \ar@{->}[rr]_-{\beta} && \cdot }}
		\quad\equiv\quad
		\vcenter{\xymatrix@1@!0@=2em{& \cdot \ar@{}[rd]|-{\tau} \ar@{->}[rr] \ar@{->}[ld] \ar[dddd]|(.25){\hole} && \cdot\ar@{->}[ld] \ar[dddd] \\
		\cdot \ar@{->}[rr] \ar[dddd] && \cdot \ar[dddd]\\
		\\\\
		& \cdot \ar@{}[rd]|-{\beta} \ar@{->}[rr]|{\hole} \ar@{->}[ld] && \cdot \ar@{->}[ld]\\
		\cdot \ar@{->}[rr] && \cdot}}
		\quad=\quad
		\vcenter{\xymatrix@1@!0@=2em{& \cdot \pullbackdots \ar@{->}[rr] \ar@{->}[ld] \ar@{-->}[dd] && \cdot\ar@{->}[ld] \ar@{-->}[dd] \\
		\cdot \pullbackdots \ar@{->}[rr] \ar@{==}[dd] && \cdot \ar@{==}[dd]\\
		& \cdot \opskewpullbackdots \ar@{.>}[rr] \ar@{.>}[ld]|-s \ar@{.>}[dd] && \cdot \opskewpullbackdots \ar@{.>}[ld]|-t \ar@{.>}[dd] \\
		\cdot \ar@{->}[rr] \ar@{->}[dd] && \cdot \ar@{->}[dd]\\
		& \cdot \ar@{->}[rr]|(.3)d|{\hole} \ar@{->}[ld]|-a && \cdot \ar@{->}[ld]|-b\\
		\cdot \ar@{->}[rr]|-c && \cdot}}
		\quad\equiv\quad
		\vcenter{\xymatrix@1@!0@=2em{\cdot \pullbackdots \ar@{->}[rr]^-{\tau} \ar@{-->}[dd] && \cdot \ar@{-->}[dd] \\\\
		\cdot \ar@{->}[rr] \ar@{->}[dd] && \cdot \ar@{->}[dd] \\\\
		\cdot \ar@{->}[rr]_-{\beta} && \cdot }}
	\]
	When we view it as a cube (second diagram), we may use Proposition~\ref{Prop DF as Pullback Up To a Pullback} to decompose it using pullbacks as on the right. We find trivial $\Gamma$-coverings $s$ and~$t$, respective pullbacks of $a$ and $b$ which lie in $\B$. Since the top back square with the dashed arrows is a pullback, the top square $\tau$ in the middle diagrams is a trivial double covering, when viewed as an arrow from left to right, because it is a pullback of a double extension between $\Gamma$-coverings, the middle arrow in the diagram on the right.

	We could, of course, have taken pullbacks of $d$ and $c$ instead, which would have shown that $\tau$ is a trivial double covering, when viewed as an arrow from back to front.
\end{proof}

\begin{definition}
	A double extension $\alpha$ in $\C$ is called a \defn{symmetric double covering (with respect to $\Gamma$)} when a discrete fibration $\tau\to \alpha$ exists whose domain $\tau$ is a symmetrically trivial double covering.
\end{definition}

\begin{theorem}\label{Theorem Covering iff Symmetric Covering}
	A double extension is a double covering if and only if it is a symmetric double covering.
\end{theorem}
\begin{proof}
	Let $\alpha$ be a double extension which is a double covering, when viewed as an arrow $\alpha\colon d\to c$. Then there is a double extension considered as an arrow $\gamma\colon e\to c$ such that $\gamma^*(\alpha)\colon f\to e$ is a trivial double covering.
	\begin{equation*}\label{G*F}
		\vcenter
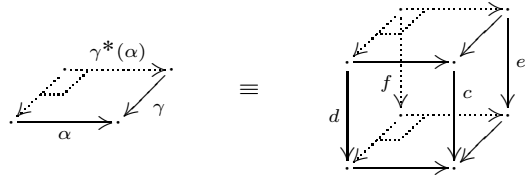
{\xymatrix@1@!0@=2em{
		& \cdot \skewpullbackdots \ar@{.>}[rr]^-{\gamma^*(\alpha)} \ar@{.>}[ld] && \cdot \ar@{->}[ld]^-\gamma\\
		\cdot \ar@{->}[rr]_-\alpha && \cdot}}
		\qquad\equiv\qquad
		\vcenter{\xymatrix@1@!0@=2em{& \cdot \skewpullbackdots \ar@{.>}[rr] \ar@{.>}[ld] \ar@{.>}[dd]_(.7)f && \cdot \ar@{->}[ld] \ar@{->}[dd]^-e \\
		\cdot \ar@{->}[rr] \ar@{->}[dd]_-{d} && \cdot \ar@{->}[dd]^(.3){c}\\
		& \cdot \skewpullbackdots \ar@{.>}[rr] \ar@{.>}[ld] && \cdot \ar@{->}[ld]\\
		\cdot \ar@{->}[rr] && \cdot}}
	\end{equation*}
	This implies that there exist $\Gamma$-coverings $y$ and $z$ such that the back square in the above cube (the square $\gamma^*(\alpha)$, now viewed as the \emph{vertical} arrow $\underline{\gamma^*(\alpha)}$) may be decomposed as a vertical composite:
	\begin{equation}\label{EqDecompTriv}
		\vcenter{\xymatrix@1@!0@=2em{\cdot \ar@{.>}[rr] \ar@{.>}[dd]_-f && \cdot \ar@{->}[dd]^-e \\\\
		\cdot \ar@{.>}[rr] && \cdot }}
		\qquad=\qquad
		\vcenter{\xymatrix@1@!0@=2em{
		\cdot \ar@{.>}[rr] \ar@{-->}[dd] \pullbackdots && \cdot \ar@{-->}[dd]\\\\
		\cdot \ar@{.>}[rr] \ar@{.>}[dd]_-z && \cdot \ar@{.>}[dd]^-y\\\\
		\cdot \ar@{.>}[rr] && \cdot}}
	\end{equation}
	By the definition of a trivial double covering, we do indeed have $\Gamma$-coverings $z'$ and~$y'$, domain and codomain of a double extension $\lambda\colon z'\to y'$ which pulls back to~$\gamma^*(\alpha)$.
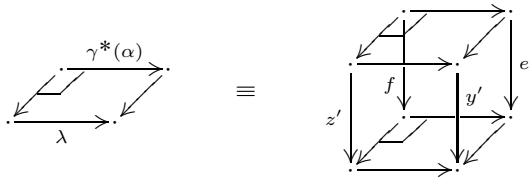
	\[
		\vcenter{\xymatrix@1@!0@=2em{
		& \cdot \skewpullback \ar@{->}[rr]^-{\gamma^*(\alpha)} \ar@{->}[ld] && \cdot \ar@{->}[ld]\\
		\cdot \ar@{->}[rr]_-\lambda && \cdot}}
		\qquad\equiv\qquad
		\vcenter{\xymatrix@1@!0@=2em{& \cdot \skewpullback \ar@{->}[rr] \ar@{->}[ld] \ar@{->}[dd]|-{\hole}_(.7)f && \cdot \ar@{->}[ld] \ar@{->}[dd]^-e \\
		\cdot \ar@{->}[rr] \ar@{->}[dd]_-{z'} && \cdot \ar@{->}[dd]^(.3){y'}\\
		& \cdot \skewpullback \ar@{->}[rr]|-{\hole} \ar@{->}[ld] && \cdot \ar@{->}[ld]\\
		\cdot \ar@{->}[rr] && \cdot}}
	\]
	In order to obtain the needed $\Gamma$-coverings $y$ and $z$, we may now take pullbacks as in the diagram
	\[
		\vcenter
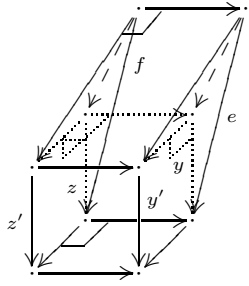
{\xymatrix@1@!0@=2em{&& \cdot\altskewpullback \ar[rr] \ar[ddddl]^(.25)f|(.75){\hole} \ar[dddll] \ar@{-->}[ldd] && \cdot \ar[ddddl]^-e \ar[dddll] \ar@{-->}[ldd]\\\\
		& \cdot \opskewpullbackdots \skewpullbackdots\ar@{.>}[rr] \ar@{.>}[ld] \ar@{.>}[dd]_(.7){z} && \cdot \opskewpullbackdots \ar@{.>}[ld] \ar@{.>}[dd]_-y \\
		\cdot \ar@{->}[rr] \ar@{->}[dd]_-{z'} && \cdot \ar@{->}[dd]^(.3){y'}\\
		& \cdot \skewpullback \ar@{->}[rr]|-{\hole} \ar@{->}[ld] && \cdot \ar@{->}[ld]\\
		\cdot \ar@{->}[rr] && \cdot}}
	\]
	and notice that the dashed comparison square between this cube and the induced dotted pullback is again a pullback.

	We come back to the vertical decomposition of $\underline{\gamma^*(\alpha)}$ in~\eqref{EqDecompTriv}. Since $y$ and $z$ are $\Gamma$-coverings, there exists a double extension $\delta$ which is such that when pulling back horizontally along it, we find trivial $\Gamma$-coverings $s$ and $t$.
	\begin{equation}\label{EqH}
		\vcenter{\xymatrix@1@!0@=2em{& \cdot \opskewpullbackdots \ar@{.>}[ld] \ar@{.>}[dd]^-\tau \\
		\cdot \ar@{->}[dd]_-{\underline{\gamma^*(\alpha)}} \\
		& \cdot \ar@{->}[ld]^-{\delta} \\
		\cdot }}
		\qquad\equiv\qquad
		\vcenter{\xymatrix@1@!0@=2em{& \cdot \pullbackdots \opskewpullbackdots \ar@{.>}[rr] \ar@{.>}[ld] \ar@{.>}[dd] && \cdot \opskewpullbackdots \ar@{.>}[ld] \ar@{.>}[dd] \\
		\cdot \pullback \ar@{->}[rr] \ar@{->}[dd] && \cdot \ar@{->}[dd]\\
		& \cdot \opskewpullbackdots \ar@{.>}[rr] \ar@{.>}[ld] \ar@{.>}[dd]_(.7){t} && \cdot \opskewpullbackdots \ar@{.>}[ld] \ar@{.>}[dd]^-s \\
		\cdot \ar@{->}[rr] \ar@{->}[dd]_-{z} && \cdot \ar@{->}[dd]^(.3){y}\\
		& \cdot \ar@{->}[rr]|{\hole} \ar@{->}[ld] && \cdot \ar@{->}[ld]\\
		\cdot \ar@{->}[rr] && \cdot}}
	\end{equation}
	To find such a $\delta$, we first consider an extension that splits $y$: the extension $y$ pulls back along it to a trivial $\Gamma$-covering $s$. We then take horizontal pullbacks in order to obtain the following cube.
	\[
		\vcenter
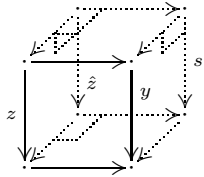
{\xymatrix@1@!0@=2em{& \cdot \opskewpullbackdots \skewpullbackdots\ar@{.>}[rr] \ar@{.>}[ld] \ar@{.>}[dd]^(.7){\hat z} && \cdot \opskewpullbackdots \ar@{.>}[ld] \ar@{.>}[dd]^-s \\
		\cdot \ar@{->}[rr] \ar@{->}[dd]_-{z} && \cdot \ar@{->}[dd]^(.3){y}\\
		& \cdot \skewpullbackdots\ar@{.>}[rr] \ar@{.>}[ld] && \cdot \ar@{.>}[ld]\\
		\cdot \ar@{->}[rr] && \cdot}}
	\]
	We then consider an extension along which to pull back the $\Gamma$-covering $\hat z$ so that we find a trivial $\Gamma$-covering $t$.
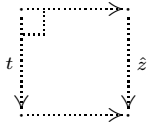
	\[
		\vcenter{\xymatrix@1@!0@=2em{\cdot \ar@{.>}[rr] \pullbackdots \ar@{.>}[dd]_-t && \cdot \ar@{.>}[dd]^-{\hat z} \\\\
		\cdot \ar@{.>}[rr] && \cdot }}
	\]
	This yields the bottom cube in~\eqref{EqH}, whose bottom square is $\delta$. The other squares are obtained by taking further pullbacks. The resulting composite double extension in the back is called $\tau$.

	From its construction, it is clear that the composed three-cubical extension from~$\tau$ to $\alpha$ is a discrete fibration. Now we only need to prove that $\tau$ is a symmetrically trivial double covering. This follows from the fact that $s$ and $t$ are trivial $\Gamma$-coverings: if we reflect them into~$\B$, then the induced comparison cubes will be pullbacks.

	For the converse, we take discrete fibrations $\beta\ot \tau \to \alpha$ where $\beta$ lies in $\B$. Lemma~\ref{Lemma Trivial} then tells us that $\tau$ is a trivial double covering, whichever way we view it as an arrow. We now pull back as in the diagram
	\[
		\vcenter
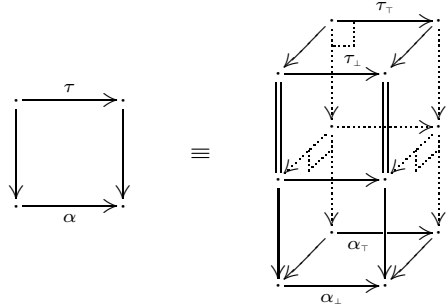
{\xymatrix@1@!0@=2em{\cdot \ar@{->}[rr]^-\tau \ar@{->}[dd] && \cdot \ar@{->}[dd] \\\\
		\cdot \ar@{->}[rr]_-{\alpha} && \cdot }}
		\qquad\equiv\qquad
		\vcenter{\xymatrix@1@!0@=2em{& \cdot \pullbackdots \ar@{->}[rr]^-{\tau_{\ttop}} \ar@{->}[ld] \ar@{.>}[dd] && \cdot\ar@{->}[ld] \ar@{.>}[dd] \\
		\cdot \ar@{->}[rr]^(.7){\tau_{\pperp}} \ar@{=}[dd] && \cdot \ar@{=}[dd]\\
		& \cdot \opskewpullbackdots \ar@{.>}[rr] \ar@{.>}[ld] \ar@{.>}[dd] && \cdot \opskewpullbackdots \ar@{.>}[ld] \ar@{.>}[dd] \\
		\cdot \ar@{->}[rr] \ar@{->}[dd] && \cdot \ar@{->}[dd]\\
		& \cdot \ar@{->}[rr]_(.25){\alpha_{\ttop}}|{\hole} \ar@{->}[ld] && \cdot \ar@{->}[ld]\\
		\cdot \ar@{->}[rr]_-{\alpha_{\pperp}} && \cdot}}
	\]
	where in the diagram on the right, the top square is $\tau$ and the bottom square is $\alpha$. Notice that the middle horizontal square is a double covering as a quotient of the double covering $\tau$. It follows that $\alpha$, viewed as an arrow from the left to the right, is a double covering as well. We could, however, use essentially the same argument to show that $\alpha$ is a double covering, when viewed as an arrow from back to front.
\end{proof}

\begin{remark}
	Notice that the very end of this argument is not valid for normal double coverings, since these are not closed under quotients---see also Remark~\ref{Remark Quotient}.
\end{remark}

\begin{corollary}\label{Corollary Double is Symmetric}
	The concept of a double covering is symmetric: it does not depend on the way the square is considered as an arrow between arrows.\noproof
\end{corollary}

\section{Higher coverings}\label{Section Higher Coverings}

There is also a symmetric interpretation of the concept of a higher covering, where the above definitions and results may essentially be copied (and used in a proof by induction). The idea is to view a covering as an $n$-dimensional \emph{object} rather than as an \emph{arrow} between $(n-1)$-dimensional objects. This improves upon a similar result in~\cite{RVdL2}, valid for abelianisation in semi-abelian categories. Recall the introduction to Section~\ref{Section Discrete Fibrations Double} where we explain what happens for (one-dimensional) $\Gamma_0$-coverings:

\begin{example}\label{Example Pullback as Disc Fib}
	A discrete fibration of order $2$ is the same thing as a pullback square of $1$-fold extensions in $\C_0$. By definition then, an extension $\alpha$ is a covering if and only if there is a span of discrete fibrations ${\beta\ot \tau\to \alpha}$ where $\beta$ is a primitive $\Gamma_0$-covering: an extension in~$\B_0$.
\end{example}

We fix a level $n\geq 2$ and view the $n$-fold extensions in the given tower of Galois structures above $\Gamma_0$ as the \emph{objects} of study. Discrete fibrations are specific morphisms of such:

\begin{definition}
	An $(n+1)$-cubical extension $\alpha$ in $\C_0$ is a \defn{discrete fibration (of order~$n+1$)} when it is a limit cube---more precisely, the cone out of the initial object $\alpha_{\meet}$ is a limit of the diagram without $\alpha_{\meet}$.
\end{definition}

We may consider a discrete fibration of order $n+1$ as an arrow between two $n$-cubical extensions in~$n+1$ non-equivalent ways. By making such a choice, the cube acquires a domain $\alpha$ and a codomain $\beta$. The entire structure is then called a \defn{discrete fibration from $\alpha$ to $\beta$}.

\begin{example}\label{Example Higher DF pullback}
	Any $(n+1)$-cubical extension, obtained as a pullback in $\C_{n-1}$ of two $n$-fold extensions, is a discrete fibration of order $n+1$: this is an instance of Example~\ref{Example DF Via Cone}.
\end{example}

\begin{example}
	Example~\ref{Example DF arrow between pullbacks} is now mimicked by the fact---which is a reformulation of Example~\ref{Example Higher DF pullback}---that any $(n+1)$-cubical extension between discrete fibrations of order $n$ is a discrete fibration of order $n+1$.
\end{example}

Alternatively, we may choose one of the $n+1$ ways an $(n+1)$-cube may be seen as a commutative square whose edges are $n$-cubes. We then obtain the following general version of Proposition~\ref{Prop DF as Pullback Up To a Pullback}.

\begin{proposition}\label{Prop DF as Pullback Up To a DF}
	An $(n+1)$-cubical extension is a discrete fibration if and only if, when viewed as a square of $n$-cubical extensions---the outer quadrangle in the diagram
	\[
		\vcenter
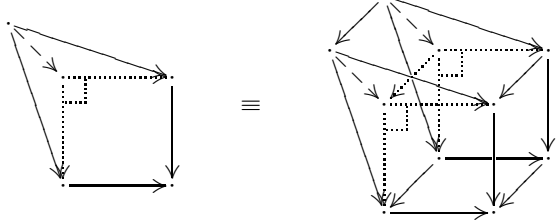
{\xymatrix@1@!0@=2em{\cdot \ar@{-->}[rd] \ar[drrr] \ar[dddr]\\
		&\cdot \pullbackdots \ar@{.>}[rr] \ar@{.>}[dd] && \cdot \ar@{->}[dd] \\\\
		&\cdot \ar@{->}[rr] && \cdot }}
		\qquad\equiv\qquad
		\vcenter{\xymatrix@1@!0@=2em{& \cdot \ar[ld] \ar@{-->}[rd] \ar[drrr] \ar[dddr]|{\hole}\\
		\cdot \ar@{-->}[rd] \ar[drrr] \ar[dddr] && \cdot \pullbackdots \ar@{.>}[rr] \ar@{.>}[ld] \ar@{.>}[dd] && \cdot \ar@{->}[ld] \ar@{->}[dd] \\
		&\cdot \pullbackdots \ar@{.>}[rr] \ar@{.>}[dd] && \cdot \ar@{->}[dd]\\
		&& \cdot \ar@{->}[rr]|{\hole} \ar@{->}[ld] && \cdot \ar@{->}[ld]\\
		&\cdot \ar@{->}[rr] && \cdot}}
	\]
	---the induced dashed comparison to the dotted pullback is not just an $n$-cubical extension, but a discrete fibration of order $n$.
\end{proposition}
\begin{proof}
	It is clear that the outer cube on the right is a limit $(n+1)$-cube if and only if the dashed comparison square, viewed as an $n$-cube, is a limit. We may use this as the induction step in a proof by induction of which Proposition~\ref{Prop DF as Pullback Up To a Pullback} forms the base step.
\end{proof}

We obtain the following generalisation of Corollary~\ref{Corollary DF compose} and Example~\ref{Visual example}.

\begin{corollary}
	Any composite of two discrete fibrations of order $n+1$, considered as arrows between $n$-cubical extensions, is again a discrete fibration.\noproof
\end{corollary}

\begin{proposition}\label{Proposition Pullback-Stability}
	Viewed as $(n+1)$-fold extensions (in any direction), discrete fibrations are pullback-stable: any pullback of a discrete fibration along an $(n+1)$-fold extension is again a discrete fibration. \noproof
\end{proposition}

Theorem~\ref{Theorem Covering iff Symmetric Covering} can be strengthened as follows. Recall that the theorem characterises a $\Gamma_{n-1}$-covering $c$ as an $n$-fold extension for which there exists a span of discrete fibrations $b\ot t\to c$ of order $n+1$ where $b$ is a primitive $\Gamma_{n-1}$-covering: it is an $n$-fold extension of which the domain $b_{\ttop}$ and codomain $b_{\pperp}$ are in $\B_{n-1}$. We now add to this, that any $(n+1)$-fold extension $\alpha\colon c\to c'$ between such $\Gamma_{n-1}$-coverings induces a morphism of spans as in Diagram~\eqref{Span of extensions} below, which happens to be a span ${\beta\ot \tau\to \alpha}$ of discrete fibrations of order $n+1$.

Note that $\alpha$ is an extension in $\B_{n+1}$, so that it is automatically a $\Gamma_{n}$-covering. Hence, Theorem~\ref{Theorem Covering iff Symmetric Covering} already yields a span of discrete fibrations ${\beta\ot \tau\to \alpha}$; however, here $\beta$ is merely again in $\B_{n+1}$. (We could actually take identity discrete fibrations.) In Lemma~\ref{Covering iff Symmetric Covering is Functorial}, the obtained span is of a different nature: it yields an $(n+1)$-fold extension $\beta$ between two primitive $\Gamma_{n-1}$-coverings $b$ and $b'$.

\begin{lemma}\label{Covering iff Symmetric Covering is Functorial}
	If $c$ and $c'$ are $\Gamma_{n-1}$-coverings (objects of $\B_n$) and $\alpha\colon c\to c'$ is an $(n+1)$-fold extension between those, then
	\begin{itemize}
		\item 	an $(n+1)$-fold extension between trivial $\Gamma_{n-1}$-coverings $\tau\colon t\to t'$, and
		\item  an $(n+1)$-fold extension $\beta\colon b\to b'$ between primitive $\Gamma_{n-1}$-coverings $b$ and $b'$, which are $n$-fold extensions whose domains ($b_{\ttop}$ and $b'_{\ttop}$) and codomains ($b_{\pperp}$ and $b'_{\pperp}$) are objects of $\B_{n-1}$
	\end{itemize}
	exist, together with (horizontal) discrete fibrations
	\begin{equation}\label{Span of extensions}
		\vcenter{\xymatrix{b \ar@{<-}[r] \ar[d]_-\beta & t \ar[r] \ar[d]^-\tau & c \ar[d]^-\alpha \\
				b' \ar@{<-}[r] & t' \ar[r] & c'}}
	\end{equation}
	of order $n+1$ between them.
\end{lemma}
\begin{proof}
	All of the constructions in the first part of the proof of Theorem~\ref{Theorem Covering iff Symmetric Covering} may be extended to ($(n+1)$-cubic) extensions between coverings by taking appropriate pullbacks when needed.
\end{proof}

\begin{remark}\label{Remark Strategy}
	Shifting the point of view in Lemma~\ref{Covering iff Symmetric Covering is Functorial} from $\alpha$ being a morphism to $\alpha$ being an object, say of dimension $n$, the statement becomes: Considering an object $\alpha$ which is a primitive $\Gamma_{n-1}$-covering, a span of discrete fibrations ${\beta\ot \tau\to \alpha}$ of order $n+1$ exists where~$\beta$ is an extension between two $\Gamma_{n-2}$-coverings. We thus manage to connect, via a span of discrete fibrations, any primitive covering to an extension between primitive coverings of a lower order. Knowing that any covering is connected, via a span of discrete fibrations, to a primitive covering, by induction we can connect any $n$-fold covering, via a zigzag of discrete fibrations---see the proof of Theorem~\ref{Main Theorem}---to an $n$-cubical diagram in the base category $\B_0$. We are now going to use these ideas to characterise coverings in the absolute, symmetrical way we are pursuing.
\end{remark}

\begin{definition}
	An $n$-cubical extension $\tau$ in $\C_0$ is called an ($n$-fold) \defn{symmetrically trivial covering} when a discrete fibration $\tau\to \beta$ exists whose codomain $\beta$ is an $n$-fold extension in $\B_0$.
\end{definition}

\begin{lemma}\label{Lemma Trivial Higher}
	An $n$-fold symmetrically trivial covering is a trivial $\Gamma_{n-1}$-covering, in whichever way this $n$-cube is considered as an arrow between $(n-1)$-cubical extensions.
\end{lemma}
\begin{proof}
	This adds an element of induction to the proof of Lemma~\ref{Lemma Trivial}. Replace any pullback in the diagram there with a discrete fibration of the appropriate order, while $s$ and $t$ are $\Gamma_{n-2}$-coverings. This forms the induction step of a proof by downwards induction on $n$, where the base step is Lemma~\ref{Lemma Trivial}, applied to the reflection of $\C_{n-2}$ to $\B_{n-2}$.
\end{proof}

\begin{definition}
	An $n$-cubical extension $\alpha$ is called an ($n$-fold) \defn{symmetric covering} when a discrete fibration $\tau\to \alpha$ exists whose domain $\tau$ is a symmetrically trivial covering.
\end{definition}

\begin{theorem}\label{Main Theorem}
	For each $n\geq 2$, an $n$-fold extension $\alpha$ is a $\Gamma_{n-1}$-covering if and only if it is a symmetric covering. More precisely, $\alpha$ is a $\Gamma_{n-1}$-covering exactly when there exists a span of discrete fibrations ${\beta\ot \tau\to \alpha}$ of order $n+1$ where $\beta$ is an $n$-cubical extension in $\B_0$. Hence, the concept of a higher covering is symmetric: it does not depend on the way the $n$-cube is considered as an arrow between $(n-1)$-cubes.
\end{theorem}
\begin{proof}
	Let $\alpha$ be a $\Gamma_{n-1}$-covering. As in Remark~\ref{Remark Strategy}, we apply Lemma~\ref{Covering iff Symmetric Covering is Functorial} repeatedly, in order to obtain a span of discrete fibrations ${\beta\ot \tau\to \alpha}$ where $\beta$ is an $n$-cubical extension in $\B_0$. We first obtain a span $\beta_n\ot \tau_n\to \alpha_n=\alpha$ where $\beta_n$ is a primitive $\Gamma_{n-1}$-covering (Theorem~\ref{Theorem Covering iff Symmetric Covering}). Next, Lemma~\ref{Covering iff Symmetric Covering is Functorial} gives us a span
	\[
		\beta_{n-1}\ot \tau'_{n-1}\to \alpha_{n-1}=\beta_n
	\]
	where $\beta_{n-1}$ is an $n$-fold extension between two primitive $\Gamma_{n-2}$-coverings. We pull back
	
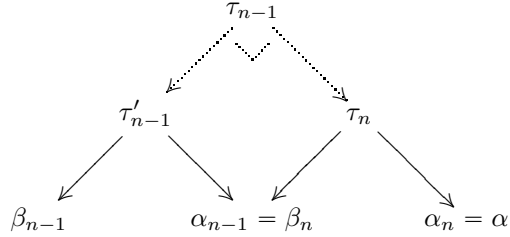
\[
		\xymatrix@!0@=4em{&& \tau_{n-1} \rotpullbackdots \ar@{.>}[dl] \ar@{.>}[dr]\\
		& \tau'_{n-1} \ar[dl] \ar[dr] && \tau_n \ar[dl] \ar[dr] \\
		\beta_{n-1} && \alpha_{n-1}=\beta_n && \alpha_n=\alpha}
	\]
	and compose to find the span $\beta_{n-1}\ot \tau_{n-1}\to \alpha$: note that here we use Proposition~\ref{Proposition Pullback-Stability}. We repeat this process until we find $\beta=\beta_1$, an $n$-cubical extension in~$\B_0$, together with an $n$-cubical extension $\tau=\tau_1$ and appropriate discrete fibrations ${\beta\ot \tau\to \alpha}$.

	The converse is an obvious variation on the second part of the proof of Theorem~\ref{Theorem Covering iff Symmetric Covering}, with Lemma~\ref{Lemma Trivial} replaced by Lemma~\ref{Lemma Trivial Higher}.
\end{proof}

\section*{Acknowledgement}
The first author wishes to thank Marino Gran, Pedro Vaz and Enrico Vitale for their kind support during the writing of this article (November 2021--June 2022).

% \bibliography{tim,symmetry}
% \bibliographystyle{amsplain}

% .bbl

\end{document}